\documentclass[12pt]{amsart}
\usepackage{graphics,wrapfig, graphicx, mathrsfs}
\usepackage{dsfont}
\usepackage{amsmath}
\usepackage{amssymb}
\usepackage{tcolorbox}
\newtheorem{lemma}{Lemma}[section]
\newtheorem{question}{Question}[section]
\newtheorem{theorem}{Theorem}[section]
\newtheorem{corollary}{Corollary}[section]

\newtheorem{proposition}{Proposition}[section]

\newtheorem{definition}{Definition}[section]
\newtheorem{remark}{Remark}[section]

\numberwithin{equation}{section} \numberwithin{theorem}{section}
\numberwithin{example}{section} \numberwithin{remark}{section}
\numberwithin{figure}{section} \numberwithin{algorithm}{section}

\usepackage[colorlinks]{hyperref}
\AtBeginDocument{
   \hypersetup{
    linkcolor=blue,
    citecolor=blue,
 }
}

\def\ep{\varepsilon}
\def\ba{\begin{array}}
\def\ea{\end{array}}
\def\bma{\left(\begin{matrix}}
\def\ema{\end{matrix}\right)}
\def\be{\begin{equation}}
\def\ee{\end{equation}}
\def\vv{{\bf v}}

\def\bH{{\bf H}}
\def\bnu{\boldsymbol{\nu}}

\begin{document}
\title[Flexibility and rigidity of 2D Euler system]{Flexibility and rigidity of steady states of the two-dimensional Euler equations in an infinite channel}

		\author{Yupei Huang}
		\address{Department of Mathematics\\
Imperial College London\\
Exhibition Road, South Kensington, London SW7 2AZ}
		\email{yhuang9@ic.ac.uk}

		\author{Chunjing Xie}
		\address{School of Mathematical Sciences, 
				Ministry of Education Key Laboratory of Scientific and Engineering Computing,
				and CMA-Shanghai, Shanghai Jiao Tong University, 800 Dongchuan Road, Shanghai, China}
		\email{cjxie@sjtu.edu.cn}
	
        \author{Chilin Zhang}
\address{School of Mathematical Sciences, Fudan University, Shanghai 200433, China}\email{zhangchilin@fudan.edu.cn}
	\begin{abstract}
     We study steady solutions to the two-dimensional incompressible Euler equations in an infinite channel, whose far-field limits are  uniformly non-stagnant shear flows. In the smooth category,for a broad class of prescribed far-field shear profiles, non-shear steady states exist via the construction of two-dimensional solutions of the semilinear elliptic equations of stream function by the  min--max method. In the analytic category, we establish a comparison principle for the analytic steady states and we show for a dense family of analytic uniformly non-stagnant shear profiles, every analytic steady state with the prescribed far field must itself be a shear flow.
In particular, there are  far-field shear profiles which exhibit flexibility in the smooth category but rigidity in the analytic category.
  Furthermore, the dense rigidity is sharp in the sense that there exists analytic shear profile which admits flexibility in the analytic category.
 
	\end{abstract}


	\maketitle
	
	\section{Introduction}
  The two-dimensional incompressible Euler equations were introduced by Leonhard Euler in 1757~\cite{euler1757principes}:
\begin{equation}\label{eqn:Euler}
\left\{
\begin{aligned}
&\partial_t \vv + \vv \cdot \nabla \vv+\nabla P = 0,\\
&\nabla\cdot \vv = 0.
\end{aligned}
\right.
\end{equation}
 The divergence-free vector field $\vv$ represents the fluid velocity and the scalar function $P$ denotes the pressure. These equations constitute a fundamental model in fluid mechanics, with applications ranging from the study of turbulence to aerodynamic flow. The global well-posedness theory for the two-dimensional incompressible Euler equations  goes back to classical works such as~\cite{holder1933unbeschrankte,wolibner1933theoreme,yudovich1963non,beale1984remarks}. 
 
On the other hand, understanding the long time behavior of the two-dimensional Euler flows remains a major open problem. Numerical and physical experiments reveal an amazing fact: generic solutions of two-dimensional incompressible Euler equation relax to much simpler state in the long time limit. It was  conjectured that this phenomenon is related to the lack of compactness of the orbits of the two-dimensional incompressible Euler equations \cite{sverak2011course,shnirelman2013long, Elgindi2026Dynamics}. 
  One mechanism leading to such a loss of compactness is phase mixing, which can produce inviscid damping\footnote{Throughout this paper, we denote inviscid damping as non-shear solutions of 2D Euler equation in the channel converge weakly to shear solutions as time goes to infinity.}; see, for example,~\cite{lin2011inviscid,bedrossian2015inviscid,WeiZhangZhao2018,zhao2024inviscid,NonlinearIonescu}. Near monotone shear flows, phase mixing transfers the non-shear component of the solution to increasingly high frequencies. Consequently, although the vorticity need not converge strongly, the associated velocity field  converge at large times to a shear flow. 
 
 In an infinite channel, solutions can lose compactness through a different mechanism: the background shear transports the non-shear component of the solution toward spatial infinity. In~\cite{GuoLuo2026}, it is shown that, in a perturbative regime near uniformly non-stagnant shear flows\footnote{ A uniformly non-stagnant shear flow is a steady state in $\mathbb{R}\times [-1,1]$ with the stream function in the form $\varphi(y)$ such that $\varphi^{'}(y)>0$, for all $y\in [-1,1]$}, this mechanism drives solutions to converge weakly to the background shear flow as time tends to infinity. By contrast, the companion work~\cite{GuoLuoQin2026} suggests that this phenomenon may fail when the background shear possesses stagnation points. Together, these results indicate that the structure of the far-field profile plays a decisive role in transporting non-shear information to spatial infinity. A natural first step toward understanding whether this relaxation mechanism persists beyond the perturbative regime is to determine whether the same far-field condition already rules out stationary non-shear structures. 
 
 It is natural to have the following question.
\begin{question}\label{question:1}
Must every steady state in a channel that converges in the far field to a uniformly non-stagnant shear flow necessarily be a shear flow?
\end{question}

The classification of steady states in a channel was first studied in~\cite{hamel2017shear}. Hamel and Nadirashvili proved that if the magnitude of the velocity is uniformly bounded away from zero, then the steady state must be a shear flow. Later, the case of steady states with stagnation points was considered in~\cite{DrivasNualart2026,LiLuShahgholianXie2026,GuiXieXu2026}. Under the additional assumption that the flow is laminar, the steady state was again shown to be necessarily a shear flow.

For two-dimensional steady Euler equations, after introducing the stream function $\psi$ satisfying $\vv =\nabla^\perp\psi$, one has
\[
\nabla^\perp\psi \cdot \nabla \Delta \psi=0.
\]
A key step in many previous works is to derive a semilinear elliptic equation for the associated stream function and then apply elliptic techniques, such as the moving-plane method, to establish symmetry. The property of admitting semilinear elliptic equation is closely connected to the rigidity of steady states; at the same time, it also provides an important tool for investigating their flexibility. For example, in~\cite{lin2011inviscid} and~\cite{coti2023stationary}, bifurcation analysis for the corresponding semilinear elliptic equations was used to construct non-shear steady states near Couette and Kolmogorov flows in a channel.

The effectiveness of this elliptic framework naturally raises the question of its generality: to what extent must the stream function of an Euler steady state satisfy a semilinear elliptic equation?
It was shown in~\cite{ElgindiHuangSaidXie_ClassificationSteadyEulerFlows_DMJ} that, in a simply connected domain supplemented with slip boundary condition, the stream function of any non-radial analytic steady state necessarily satisfies a semilinear elliptic equation. This rigidity appears to be specific to the analytic category. Indeed, it was proved in~\cite{elgindi2026flexibility} that there exists a broad class of smooth Morse steady states in simply connected domains whose stream functions do not satisfy any semilinear elliptic equation. One can also see ~\cite{gomez2021existence,enciso2024smooth} for other counterexamples where the stream function of the steady state fails to solve a semilinear elliptic equation.
A similar distinction arises for steady states converging in the far field to uniformly non-stagnant shear flows. In the analytic category, the stream function necessarily satisfies a semilinear elliptic equation determined by the prescribed far-field shear flow, whereas this conclusion need not hold in the smooth category. This difference in the semilinear elliptic equation structure ultimately leads to rigidity in the analytic category and flexibility in the smooth category for Question \ref{question:1}. 
\subsection{Main Result}
We consider the steady Euler flow in the channel $\Omega:=\mathbb{R}\times (-1,1)$ converging asymptotically to strictly monotonically increasing shear flow. 
In terms of the stream function $\psi$,  we consider the problem
\begin{equation}\label{eqn:Steady in channel}
    \begin{aligned}
          &\nabla^\perp\psi \cdot \nabla \Delta \psi=0,\quad \text{in}\,\, \Omega,\\&
          \psi(x,1)=\varphi(1), \quad  \psi(x,-1)=\varphi(-1),\\&    \lim_{|x|\rightarrow+\infty}\|\psi(x,\cdot)-\varphi(\cdot)\|_{C^2([-1,1])}=0,
    \end{aligned}
\end{equation}
where $\varphi\in C^\infty([-1,1])$ is a monotone increasing function.
In the following theorem, we show there are non-shear smooth solution for \eqref{eqn:Steady in channel} under mild conditions on $\psi$, while in the analytic class, there are only shear solutions for \eqref{eqn:Steady in channel} for generic monotonically increasing analytic shear flows.   
\begin{theorem}\label{MainTheorem}
For every function $\varphi\in C^{\infty}[-1,1]$ satisfying
    \begin{equation}\label{eq. additional assumption on smooth varphi}
        \varphi'(y)>0,\quad \varphi'''(y)>0,\quad\mbox{and }\varphi'(y)\varphi''''(y)<\varphi''(y)\varphi'''(y),\quad\mbox{for all }y\in[-1,1],
    \end{equation} there exists a non-shear smooth(smooth in $\overline{\Omega}$) solution to \eqref{eqn:Steady in channel}. 
    On the other hand, there exists a family functions $\varphi$, dense in the topology in $C^{\omega}([-1,1])\cap\{\varphi|\varphi^{'}>0\}$, such that every analytic solution (analytic in $\overline{\Omega}$) of \eqref{eqn:Steady in channel} is necessarily a shear flow.
\end{theorem} 

There are a few remarks in order. 

\begin{remark}
    The dense rigidity is sharp in the sense that there does exist non-stagnant analytic shear profile to which there exists an analytic non-shear solution convergent at far fields, see Proposition \ref{lem. analytic flexibility}. 
\end{remark}

\begin{remark}\label{rmk. t+1-e^-t}
We emphasize that there exists an analytic shear flow without stagnation points for which every analytic steady state with the same far-field profile is necessarily a shear flow, whereas, in the smooth category, there exists a non-shear steady states with this prescribed far field. We give a concrete  example as follows. 
    Let
\[
\Phi_0(t)=\frac{t^2}{2}+t+e^{-t}-1,
\]
and $\varphi$ solve
\[
\varphi'(y)=\sqrt{2\Phi_0(\varphi(y))},
\qquad
\varphi(-1)=1.
\]
It is direct to verify that the $\varphi$ satisfy \eqref{eq. additional assumption on smooth varphi}.  Hence there exist a smooth non-shear steady state with far field $\varphi$. In the same time, based on the discussion on Section 4, any analytic solution to \eqref{eqn:Steady in channel} should necessarily satisfy $\Delta \varphi= \Phi_0'(\varphi)$ in $\Omega$. As $\Phi_0^{''}>0$, the maximum principle implies that $\varphi$ is the unique analytic solution to \eqref{eqn:Steady in channel}.
\end{remark}

\begin{remark}[Interpretation of the assumptions]
Since $\varphi'>0$, we may define a function $F_\varphi$ on
${\mathcal{R}}_\varphi:=\varphi([-1,1])$
by
\begin{equation}\label{ODE:varphi}
F_\varphi(\varphi(y))=\varphi''(y).
\end{equation}
Straightforward calculations give
\[
F_\varphi'(\varphi(y))
=
\frac{\varphi'''(y)}{\varphi'(y)}
\quad
\text{and}
\quad 
F_\varphi''(\varphi(y))
=
\frac{
\varphi'(y)\varphi''''(y)
-
\varphi''(y)\varphi'''(y)
}{
\bigl(\varphi'(y)\bigr)^3
}.
\]
Therefore, the assumption~\eqref{eq. additional assumption on smooth varphi}
implies that
\[
F_\varphi'>0
\qquad\text{and}\qquad
F_\varphi''<0
\quad\text{on }{\mathcal{R}}_\varphi.
\]
In particular, the energy functional associated with the semilinear
equation is differentiable near $\varphi$, and $\varphi$ is a local
minimizer for the energy functional in the desired topology.
\end{remark}
\subsection{Ideas of the proofs}

The flexibility and rigidity results arise from two different ways in which
the prescribed far-field profile determines the vorticity--stream-function
relation.

We first explain the key ideas to study the flexibility. Since $\varphi'>0$, the far-field
profile determines a nonlinearity $F_{\varphi}$ on the interval
${\mathcal{R}}_\varphi=\varphi([-1,1])$
through the relation \eqref{ODE:varphi}.
However, the prescribed far field contains no information about the values of
$F_{\varphi}$ outside ${\mathcal{R}}_{\varphi}$. In the smooth category, we may therefore
extend $F_{\varphi}$ beyond this interval without changing the far-field shear
flow. We choose an extension so that the corresponding semilinear
elliptic equation possesses a suitable variational structure.

As explained in Remark~1.2, the assumptions on $\varphi$ ensure that the
shear profile is a local minimizer of the associated energy functional. At
the same time, the extension of $F_{\varphi}$ is chosen so that there exists
a competitor with lower energy, separated from the shear profile by an energy
barrier. This is precisely the setting in which a min--max argument can be
applied.

More precisely, for the energy functional $I$ defined on a
function space $\overline{\mathcal{C}}\subset H_0^1(\Omega)$ has a local minimizer $\psi=0$, and let
$\psi=\mathfrak{e}\in\overline{\mathcal{C}}$ be a competitor lying beyond an energy barrier. We
consider the class of continuous paths
\begin{equation*}
\Upsilon
:=
\Big\{\gamma\in C([0,1],\overline{\mathcal{C}})
\Big|
\,\gamma(0)=0,\,\,
\gamma(1)=\mathfrak{e}\Big\}
\end{equation*}
and define the min--max level
\begin{equation*}
\sigma
:=
\inf_{\gamma\in\Upsilon}
\max_{s\in[0,1]}
I(\gamma(s)).
\end{equation*}
Under an appropriate compactness condition, this construction produces a
critical point of $I$ at the level $\sigma$ (See, for example, \cite{struwe1980infinitely,rabinowitz1986minimax}). In our setting, the
resulting critical point is distinct from the shear profile and gives a
non-shear solution of the associated semilinear elliptic equation for the stream function.

The main difficulty is that the channel is unbounded in the horizontal
direction. Consequently, the energy functional is invariant under horizontal
translations, and a min--max sequence may drift toward spatial infinity. The
essential compactness step is therefore to recenter the sequence by suitable
horizontal rearrangement and prove that the rearranged  sequence converges to a
nontrivial critical point. Furthermore, for a particular chosen analytic profile, all the above procedures still hold under the analytic extension. Hence for such particular shear profile, the flexibility also holds in the analytic setting (see Proposition \ref{lem. analytic flexibility}).

In the
analytic category, the prescribed far-field profile determines not only the
nonlinearity on ${\mathcal{R}}_{\varphi}$, but also the entire semilinear elliptic equation
satisfied by the stream function. Thus, the freedom to modify the
vorticity--stream-function relation outside the range of the far-field
profile disappears. For a dense family of analytic shear profiles, the
resulting equation satisfies a comparison principle, which forces every
analytic steady state with the prescribed far field to coincide with the
corresponding shear flow.

\subsection{Organization of the paper}
Section \ref{sec2}  devotes to the construction of non-shear flows in a channel when the flows have the same prescribed far-field asymptotic state, under the assumption on the existence of Palais-Smale sequence. Section \ref{Section3} focuses on the technical part to construct the Palais-Smale sequence via the heat flows. The dense rigidity of the flows is proved in Section \ref{Section4}. There are two appendices of the paper. Appendix \ref{AppendixA} establishes Sobolev inequalities for functions in the channel. Existence, uniqueness, and several properties of solutions to the heat equation are established in Appendix \ref{AppendixB}.  

\section{Flexibility of smooth steady states in the channel}\label{sec2}
For given $\varphi\in C^\infty ([-1,1])$ satisfying \eqref{eq. additional assumption on smooth varphi}, there exists 
 a function $F$ defined on ${\mathcal{R}}_\varphi=\varphi([-1,1])$ such that $\varphi''=F(\varphi)$.
In this section, we intend to find a non-shear steady state in the channel $\Omega=\mathbb{R}\times(-1,1)$ satisfying
\begin{equation}\label{eq. Euler equation PROBLEM with a single F}
    \left\{\begin{aligned}
        &\Delta\psi=F(\psi)\mbox{ in }\Omega,\\
    &\|\psi(x,\cdot)-\varphi(\cdot)\|_{C^2([-1,1])}\to 0\mbox{ as }|x|\to+\infty,\\
    &\psi(x,\pm1)\equiv\varphi(\pm1),
    \end{aligned}\right.
\end{equation}
where $F$ could be an extension of the one defined on ${\mathcal{R}}_\varphi$.  
In order to construct $\psi$ in the form \eqref{eq. construct psi from u} that satisfies $\Delta\psi=F(\psi)$,  it suffices to prove the existence of nontrivial solution for the equation
\begin{equation*}
    \Delta u=F(\varphi+u)-F(\varphi)=\Phi'(\varphi+u)-\Phi'(\varphi),
\end{equation*}
where $\Phi$ is an anti-derivative of $F$.

Since $\varphi$ depends only on $y$, one can define the following potential function that depends only on $u$ and $y$:
\begin{equation}\label{eq. G and Phi relation}
    G(u,y)=\Phi(\varphi+u)-\Phi(\varphi)-\Phi'(\varphi)\cdot u.
\end{equation}
Assume that $u\in H^{1}_{0}(\Omega)$ is a non-trivial (weak) solution to
\begin{equation}\label{eq. Euler-Lagrange}
    \Delta u=G'(u,y),
\end{equation}
where we always refer to $G'(u, y)$ and $G^{(k)}(u,y)$ as the first order partial derivative with respect to the variable $u$,
then it is easy to verify that 
\begin{equation}\label{eq. construct psi from u}
 \psi(x,y):=  u(x,y)+\varphi(y).
\end{equation} 
gives a (weak) nonshear solution to the problem \eqref{eq. Euler equation PROBLEM with a single F}.

In this paper, we focus on studying the existence of a non-trivial solution to \eqref{eq. Euler-Lagrange}, given that the potential $G(u,y): \bar{\Omega} \to\mathbb{R}$ is a $C^{3}$ function satisfying the following conditions.
\begin{itemize}
    \item[(G1)] \textbf{Convexity at $0$:} We require that
    \begin{equation*}
        G(0,y)=G'(0,y)=0\mbox{ for all }y\in[-1,1]\quad\mbox{and }\inf_{y\in[-1,1]}G''(0,y)>0.
    \end{equation*}
     In the rest of the paper, we use $G^{(k)}(u, y)$ to denote the $k$-th partial derivatives of $G$ with respect to the variable $u$.
    \item[(G2)] \textbf{Asymptotic behavior at $+\infty$:} $G(u,y)$ satisfies
    \begin{equation*}
\limsup_{u\to+\infty}\sup_{y\in[-1,1]}\frac{G(u,y)}{u^{3}}<0\quad\mbox{and }\liminf_{u\to+\infty}\inf_{y\in[-1,1]}\frac{G(u,y)}{u^{3}}>-\infty.
    \end{equation*}
   
    \item[(G3)] \textbf{Derivative bounds:} There exists a uniform constant $C_{0}>0$, such that for all $(u,y)\in\mathbb{R}\times[-1,1]$, one has
    \begin{equation*}
        |G'(u,y)|\leq C_{0}\cdot(|u|+|u|^{2}),\quad|G''(u,y)|\leq C_{0}\cdot(1+|u|),\quad\mbox{and }G'''(u,y)\leq-\frac{1}{C_{0}}.
    \end{equation*}
\end{itemize}

The following existence result plays a crucial role in proving the flexibility.
\begin{proposition}[Existence of a non-trivial weak solution]\label{prop. exist weak solution}
If the function $G(u,y)$ satisfies the assumptions (G1)-(G3), then there exists a non-zero solution $u\in H^{1}_{0}(\Omega)$ satisfying $\Delta u=G'(u,y)$ in the weak sense, i.e.,
\[
\int_\Omega \nabla u\cdot \nabla \xi +G'(u, y)\xi dx =0 \quad \text{for any }\xi \in H_0^1(\Omega).
\]
\end{proposition}

The proof of Proposition \ref{prop. exist weak solution} will be given at the end of this section based on the existence of Palais-Smale sequence. Now we use Proposition \ref{prop. exist weak solution} to prove the flexibility result in Theorem \ref{MainTheorem}.

First, we show that the nontrivial solutions of \eqref{eq. Euler-Lagrange} constructed in Proposition~\ref{prop. exist weak solution} have the following higher order regularity.

\begin{lemma}\label{lem. higher order regularity}
    Assume that $\varphi(y)$ is a classical solution to $\varphi''(y)=F(\varphi)$ for $y\in[-1,1]$. Let $\Phi$ be an anti-derivative of $F$, and define $G(u,y)$ as in \eqref{eq. G and Phi relation}. If $u\in H^{1}_{0}(\Omega)$ is a weak solution to \eqref{eq. Euler-Lagrange}, then $\psi=\varphi+u$ defined as in \eqref{eq. construct psi from u} is a classical solution to \eqref{eq. Euler equation PROBLEM with a single F}. Moreover, if $F$ is smooth or analytic, then $\psi$ belongs to $C^{\infty}(\Omega)$ or $C^{\omega}(\Omega)$, respectively.
\end{lemma}
\begin{proof}
    It follows from the assumptions (G1)-(G3) and Lemma~\ref{lem. Sobolev in a channel} that both $u$ and $G'(u,y)$ belong to $L^{2}(\Omega)$. Applying the standard $W^{2,p}$ theory yields that
    \begin{equation*}
        u\in W^{2,2}(\Omega\cap B_{R})\subseteq L^{\infty}(\Omega\cap B_{R}),\quad\mbox{for all }R>0.
    \end{equation*}
   If $F$ is $C^\infty$ smooth or analytic, then $G'(u,y)$ is $C^\infty$ smooth or analytic in $u$, respectively. It then follows from the Schauder estimate and Morrey theorem in \cite{Morrey1958Analyticity} that $u\in C^{\infty}_{loc}(\Omega)$ or $u\in C^{\omega}_{loc}(\Omega)$, respectively. In particular, $u$ satisfies \eqref{eq. Euler-Lagrange} in the classical sense.

    The straightforward computations show that $\psi=u+\varphi$ is a classical solution of the PDE
    \begin{equation*}
        \Delta\psi=F(\psi).
    \end{equation*}
    Furthermore, $\psi(x,y)$ is $C^\infty$ smooth or analytic and is not identical to $\varphi(y)$.
    
    Since $u\in H^{1}_{0}(\Omega)\cap C^{\infty}_{loc}(\Omega)$, one has $u(x,\pm1)\equiv0$ in the classical sense. Thus $\psi(x,\pm1)=\varphi(\pm1)$. As both $u$ and $G'(u,y)$ belong to $L^{2}(\Omega)$, one has
    \begin{equation*}
        \lim_{|L|\to\infty}\Big\{\|u\|_{L^{2}([L-2,L+2]\times[-1,1])}+\|G'(u,y)\|_{L^{2}([L-2,L+2]\times[-1,1])}\Big\}=0.
    \end{equation*}
    It follows from the $W^{2,p}$ estimate and the Schauder estimate that
    \begin{equation*}
        \lim_{L\to\infty}\|u\|_{C^{2,\alpha}(\Omega_{L}^{c})}=0,\quad\mbox{where }\Omega_{L}^{c}=\overline{\Omega}\setminus([-L,L]\times[-1,1]).
    \end{equation*}
    This implies
    \begin{equation*}
        \lim_{|x|\to+\infty}\|\psi(x,\cdot)-\varphi(\cdot)\|_{C^{2,\alpha}([-1,1])}=0.
    \end{equation*}
    Therefore, $\psi$ is a non-shear classical solution to \eqref{eq. Euler equation PROBLEM with a single F}.
\end{proof}
Now, by assuming Proposition~\ref{prop. exist weak solution}, we give the proof of the flexibility part of Theorem~\ref{MainTheorem}.

\begin{proof}[Proof of Part (i) of Theorem~\ref{MainTheorem}]
    For the given function $\varphi\in C^{\infty}([-1,1])$ that satisfies $\varphi'(y)>0$ in $[-1,1]$, it follows from the implicit function theorem that there exists a smooth function $F$ in the range of $\varphi$, such that $\varphi''(y)=F(\varphi)$ for $y\in[-1,1]$. It follows from the assumption \eqref{eq. additional assumption on smooth varphi} that
    \begin{equation*}
        F'(\varphi(y))>0\mbox{ and }F''(\varphi(y))<0\quad\mbox{for all }y\in[-1,1].
    \end{equation*}

    Now let us extend $F(t)$ from ${\mathcal{R}}_\varphi=[\varphi(-1),\varphi(1)]$ to $\mathbb{R}$, such that $ F\in C^{\infty}(\mathbb{R})$ satisfies  
   \begin{equation*}
        F''(s) =-1,\text{ $\forall s \in (\varphi(-1)-1,\varphi(1)+1)^{c}$}.
    \end{equation*}
and 
    \begin{equation*}
        -C\leq F^{''}(s)\leq \frac{-1}{C}<0,\text{ $\forall s \in \mathbb{R}$,}
    \end{equation*} 
where $C>0$ is a constant.   
    
    Let $\Phi\in C^{\infty}(\mathbb{R})$ be an anti-derivative of the smooth extension of $F$ mentioned above, { as $G''(0,y)=F'(\varphi)>0$ for $y\in[-1,1]$ and $G'''(u,y)=F''(\varphi+u)\in[-C,-\frac{1}{C}]$, it follows that $G(u,y)$ satisfies all assumptions (G1)-(G3). We can apply Proposition~\ref{prop. exist weak solution} to construct a nonzero solution $u\in H_0^1(\Omega)$ satisfying \eqref{eq. Euler-Lagrange}.}

    Applying Lemma~\ref{lem. higher order regularity} yields that $\psi=\varphi+u$ is a classical solution to \eqref{eq. Euler equation PROBLEM with a single F}. Moreover, one has $\psi\in C^{\infty}_{loc}(\Omega)$ as $F(t)\in C^{\infty}(\mathbb{R})$. This finishes the proof of the first part of Theorem~\ref{MainTheorem}.
\end{proof}
\begin{remark}
If $F(t)=\Phi'(t)=t+1-e^{-t}$ (the same as the one in Remark~\ref{rmk. t+1-e^-t}), then Figure~\ref{fig: draw to functions} gives a possible non-analytic extension of $F$ that is used in the proof of Part (i) of Theorem~\ref{MainTheorem}.
    \begin{figure}[h!]\label{fig: draw to functions}
 \caption{A non-analytic extension of $F$}
    \centering
\includegraphics[width=0.5\linewidth]{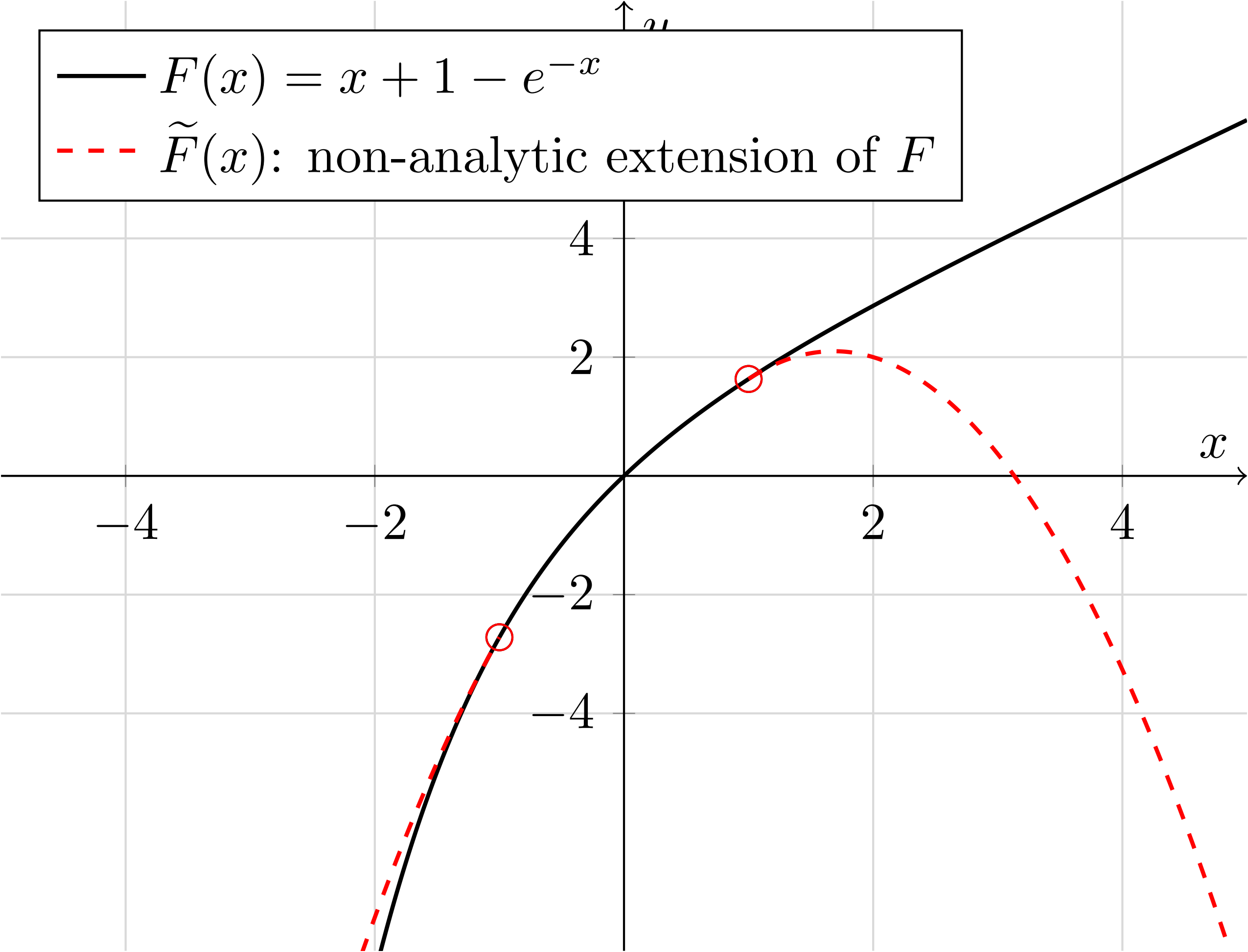}
\end{figure} 
\end{remark}

Similar to the above proof for flexibility part of Theorem~\ref{MainTheorem} , one can also prove the following analytic flexibility result.
\begin{proposition}[Analytic flexibility]\label{lem. analytic flexibility}
    There exists an analytic function $\varphi(y)$ satisfying $\varphi'>0$ in $[-1,1]$, such that the problem \eqref{eq. Euler equation PROBLEM with a single F} has a non-shear analytic solution $\psi$.
\end{proposition}
\begin{proof}
 Let $\varphi:[-1,1]\to\mathbb{R}$ be the solution to the equation
\begin{equation}\label{eq. good epsilon and f}
    \left\{\begin{aligned}
        &\varphi'(y)=\sqrt{2(\varphi^{2}-\varphi^{3})},\quad\mbox{for }y\in(-1-\delta,1+\delta),\\
        &\varphi(-1)=\epsilon,
    \end{aligned}\right.
\end{equation}
where $\epsilon>0$ is a small constant.
 Since $0<\varphi'\leq\sqrt{2}\varphi$ for all $0<\varphi<1$, one has $\varphi(y)\leq e^{\sqrt{2}(y+1)}\cdot\epsilon$. Therefore, there exists a $\delta>0$ such that  
\begin{equation*}
    \varphi'(y)>0\mbox{ and }0<\varphi(y)<\frac{1}{6}\mbox{ in }[-1-\delta,1+\delta].
\end{equation*}
By differentiating \eqref{eq. good epsilon and f}, $\varphi(y)$ is a one-dimensional analytic solution to the semilinear elliptic equation $\varphi''=2\varphi-3\varphi^{2}$. 

Next, let $F(t)=2t-{3t^2}$ and $G(u,y)=(1-3\varphi)u^{2}-u^{3}$.
The straightforward computations show that $G$ satisfies (G1)-(G3). Applying Proposition~\ref{prop. exist weak solution} and Lemma~\ref{lem. higher order regularity} yields the existence of non-zero classical solution $u$ satisfying \eqref{eq. Euler-Lagrange}. Hence $\psi=\varphi+u$ is a classical solution to \eqref{eq. Euler equation PROBLEM with a single F} with $F(t)=2t-3t^2$.  Since $F(t)\in C^{\omega}(\mathbb{R})$, it follows from Lemma ~\ref{lem. higher order regularity} that $\psi\in C^{\omega}(\overline{\Omega})$. Hence the proof of Proposition~\ref{lem. analytic flexibility} is completed.
\end{proof}

In order to prove Proposition~\ref{prop. exist weak solution}, we use the variational method, more precisely the min-max method, to construct a saddle critical point of an energy function.

\subsection{Formulation of the variational problem}
Consider the energy functional
\begin{equation}\label{eq. I(u) definition}
    I(u)=\int_{\Omega}\Big\{\frac{|\nabla u|^{2}}{2}+G(u,y)\Big\}dxdy,\quad u\in H^{1}_{0}(\Omega).
\end{equation}
If $u\in H^{1}_{0}(\Omega)$ is a critical point of \eqref{eq. I(u) definition}, then it formally satisfies \eqref{eq. Euler-Lagrange}.

In this subsection, we derive the variational formula of \eqref{eq. I(u) definition} and rigorously establish the connection between the functional $I$ and the semilinear elliptic equation \eqref{eq. Euler-Lagrange}.

Let us first study some properties of the first variation of  $I(u)$.
\begin{lemma}\label{lem. derivative of I(u), i.e.: EL equation}
    Assume that the potential function $G(u,y)$ satisfies the assumptions (G1)-(G3). Then the functional $I(\cdot)$ is $C^{1}$ in the space $H^{1}_{0}(\Omega)$. Moreover, let $u\in H^{1}_{0}(\Omega)$, the first variation of $I$ at $u$ in $H^{1}_{0}(\Omega)$, which is denoted by $\nabla I$, extends to be a linear functional in $H^{1}_{0}(\Omega)$, with
    \begin{equation}\label{eq. weak form of nabla I}
        \langle\nabla I(u),\xi\rangle=\int_{\Omega}\Big\{\nabla u\cdot\nabla\xi+G'(u,y)\xi\Big\}dxdy.
    \end{equation}
    Furthermore,
    \begin{equation}\label{eq. nabla I(u) upper bound}
        \|\nabla I(u)\|_{H^{-1}(\Omega)}\leq C\Big(\|u\|_{H^{1}(\Omega)}\Big).
    \end{equation}
\end{lemma}
\begin{proof}
    If $u\in H^{1}_{0}(\Omega)$, then it follows from (G1)-(G3) and Lemma~\ref{lem. Sobolev in a channel} that $G'(u,y)\in L^{2}(\Omega)$ for all $u\in H^{1}_{0}(\Omega)$. Then, for every $\xi\in H^{1}_{0}(\Omega)$, one has
    \begin{align*}
        \langle\nabla I(u),\xi\rangle=&\frac{d}{dt}I(u+t\xi)\Big|_{t=0}=\frac{d}{dt}\int_{\Omega}\Big\{\frac{|\nabla(u+t\xi)|^{2}}{2}+G(u+t\xi,y)\Big\}dxdy\Big|_{t=0}\\
        =&\int_{\Omega}\Big\{\nabla u\cdot\nabla\xi+G'(u,y)\xi\Big\}dxdy.
    \end{align*}
    This verifies \eqref{eq. weak form of nabla I}. It follows from the Cauchy-Schwarz inequality that
    \begin{equation*}
        \Big|\int_{\Omega}\Big\{\nabla u\cdot\nabla\xi+G'(u,y)\xi\Big\}dxdy\Big|\leq\|u\|_{H^{1}(\Omega)}\|\xi\|_{H^{1}(\Omega)}+\|G'(u,y)\|_{L^{2}(\Omega)}\|\xi\|_{L^{2}(\Omega)}.
    \end{equation*}
    Therefore, it holds that
    \begin{equation*}
        |\langle\nabla I(u),\xi\rangle|\leq\Big(\|u\|_{H^{1}(\Omega)}+\|G'(u,y)\|_{L^{2}(\Omega)}\Big)\cdot\|\xi\|_{H^{1}(\Omega)}.
    \end{equation*}
    Note that Lemma~\ref{lem. Sobolev in a channel} implies that $u\in L^{p}(\Omega)$ for $2\leq p\leq4$. It follows from the condition (G3) that
    \[
    \Big(G'(u, y)\Big)^{2}\leq C (|u|^2+|u|^4),
    \]
 $\|G'(u, y)\|_{L^{2}(\Omega)}$ must be finite. Therefore, one has
    \begin{equation*}
        |\langle\nabla I(u),\xi\rangle|\leq C(\|u\|_{H^{1}(\Omega)})\cdot\|\xi\|_{H^{1}(\Omega)}.
    \end{equation*}
    Hence, $\nabla I(u)\in H^{-1}(\Omega)$ and \eqref{eq. nabla I(u) upper bound} holds. Finally, it follows from the expression \eqref{eq. weak form of nabla I}, the Sobolev inequality \ref{eq. L p sobolev desired inequality}, and the assumption (G3) on $G''$ that $\nabla I(\cdot)$ is $H^{-1}(\Omega)$ functional continuous in $H^{1}_{0}(\Omega)$, i.e., $I(\cdot)$ is $C^{1}$ in $H^{1}_{0}(\Omega)$.
\end{proof}
Integrating \eqref{eq. weak form of nabla I} by parts yields the following alternative form of $\nabla I$,
\begin{equation}\label{eq. L2 way to understand nabla I}
    \langle\nabla I(u),\xi\rangle=\int_{\Omega}\xi\cdot\Big(G'(u,y)-\Delta u\Big)dxdy,\quad\forall\xi\in H^{1}_{0}(\Omega).
\end{equation}

the following lemma is the key to constructing a solution to \eqref{eq. Euler-Lagrange}.
\begin{lemma}\label{lem. convergence of nabla I}
    Assume that $G(u,y)$ satisfies (G1)-(G3) and define $I(u)$ as in \eqref{eq. I(u) definition}. Let $\{u_{n}\}\subseteq H^{1}_{0}(\Omega)$ converge to $u_{\infty}$ strongly in $H^{1}(\Omega)$. If $\|\nabla I(u_{n})\|_{H^{-1}(\Omega)}\to0$, then $u_{\infty}$ satisfies \eqref{eq. Euler-Lagrange} in the weak sense.
\end{lemma}
\begin{proof}
    Let $\xi\in H^{1}_{0}(\Omega)$ with $\|\xi\|_{H^{1}_{0}(\Omega)}=1$ be a test function. It follows from \eqref{eq. weak form of nabla I} that
    \begin{equation}\label{eq. converge after test 1}
        \Big|\int_{\Omega}\Big\{\nabla u_{n}\cdot\nabla\xi+G'(u_{n},y)\xi\Big\}dxdy\Big|=|\langle\nabla I(u_{n}),\xi\rangle|\leq\|\nabla I(u_{n})\|_{H^{-1}(\Omega)}\to0.
    \end{equation}
    By the assumption (G3), one has
    \begin{equation*}
       \Big|G'(u_{n},y)-G'(u_{\infty},y)\Big|\leq C(1+|u_{n}|+|u_{\infty}|)\cdot|u_{n}-u_{\infty}|.
    \end{equation*}
    Applying the Cauchy-Schwarz inequality gives
    \begin{align*}
        &\int_{\Omega}\Big|G'(u_{n},y)-G'(u_{\infty},y)\Big|^{2}dx\\
        \leq&C\|u_{n}-u_{\infty}\|_{L^{2}(\Omega)}^{2}+C\Big(\|u_{n}\|_{L^{4}(\Omega)}^{2}+\|u_{\infty}\|_{L^{4}(\Omega)}^{2}\Big)\cdot\|u_{n}-u_{\infty}\|_{L^{4}(\Omega)}^{2}.
    \end{align*}
    As $u_{n}\to u_{\infty}$ in $H^{1}(\Omega)$, one has $u_{n}\to u_{\infty}$ in $L^{p}(\Omega)$ for all $p\geq2$. Hence it follows that
    \begin{equation}\label{eq. converge after test 2}
        \Big|\int_{\Omega}\Big\{(\nabla u_{n}-\nabla u_{\infty})\cdot\nabla\xi+\Big(G'(u_{n},y)-G'(u_{\infty},y)\Big)\xi\Big\}dxdy\Big|\to0.
    \end{equation}
    Combining \eqref{eq. converge after test 1} and \eqref{eq. converge after test 2} yields
    \begin{equation*}
        \int_{\Omega}\Big\{\nabla u_{\infty}\cdot\nabla\xi+G'(u_{\infty},y)\xi\Big\}dxdy=0.
    \end{equation*}
    This implies that $u_{\infty}$ satisfies \eqref{eq. Euler-Lagrange} in the weak sense.
\end{proof}

\subsection{Outline of the min-max construction}
The major difficulty in constructing a non-trivial critical point comes from the translation invariance of the energy functional $I$, so that the sequence of functions can lose compactness due to shifting to spatial infinity. Our key idea is to perform variational analysis in a function space that is stable under the Steiner symmetrization.

Define
\begin{equation}\label{defmathcalC}
    \mathcal{C}:=\Big\{u\in C^{\infty}_{0}(\Omega):u\geq0\mbox{ in }\Omega,\ u(x,y)\equiv u(-x,y),\ \partial_{x}u\leq0\mbox{ for }x\geq0\Big\},
\end{equation}
and let $\overline{\mathcal{C}}$ be the closure of $\mathcal{C}$ in $H^{1}_{0}(\Omega)$. In this paper, we restrict our study of the energy functional $I(u)$ to $\overline{\mathcal{C}}$.

\begin{proposition}\label{prop. Existence and convergence of Palais-Smale sequence}
Assume that $G(u,y)$ satisfies the assumptions (G1)-(G3), and let $I$ be defined in \eqref{eq. I(u) definition}. Then the following statements hold.
\begin{itemize}
    \item[(1)] There exists a positive constant $\sigma$ and a sequence $\{u_{n}\}\subseteq\mathcal{C}$ such that 
    \[
    \sigma/2<I(u_n)<2\sigma \text{ and }\|\nabla I(u_{n})\|_{H^{-1}(\Omega)}\to0.
    \]
    \item[(2)] Up to a subsequence (still labelled by $\{u_{n}\}$), there exists a $u_\infty\in\overline{\mathcal{C}}$ such that  $\|u_{n}-u_{\infty}\|_{H^{1}(\Omega)}\to 0$.
\end{itemize}
\end{proposition}

With Proposition~\ref{prop. Existence and convergence of Palais-Smale sequence} and Lemma~\ref{lem. convergence of nabla I}, we are able to prove Proposition~\ref{prop. exist weak solution}.

\begin{proof}[Proof of Proposition~\ref{prop. exist weak solution}]
    Let $\{u_{n}\}$ be the sequence provided by Proposition~\ref{prop. Existence and convergence of Palais-Smale sequence} which converges to $u_\infty\in H^{1}_{0}(\Omega)$. As $\|u_{n}-u_\infty\|_{H^{1}(\Omega)}\to0$, it follows from Lemma~\ref{lem. Sobolev in a channel} that $\|u_{n}-u_\infty\|_{L^{p}(\Omega)}\to0$ (and thus $u_\infty\in L^{p}(\Omega)$) for all $p\geq2$. Therefore, it holds that
    \begin{equation*}
 I(u_\infty)=\lim_{n\to\infty}I(u_{n})\geq\frac{\sigma}{2}>0.
    \end{equation*}
    Hence, $u_\infty(x,y)\not\equiv0$ for $(x,y)\in\Omega$. Applying Lemma~\ref{lem. convergence of nabla I} yields that $u=u_\infty$ is a weak solution to \eqref{eq. Euler-Lagrange}. This finishes the proof of Proposition~\ref{prop. exist weak solution}.
\end{proof}

In the next section, we will give a detailed proof of Proposition~\ref{prop. Existence and convergence of Palais-Smale sequence}.

\section{Construction of the Palais-Smale sequence and min-max method}\label{Section3}
This section is devoted to the proof of Proposition~\ref{prop. Existence and convergence of Palais-Smale sequence}.
\subsection{Outline of the construction}
In this subsection, we outline the proof of the first part of Proposition~\ref{prop. Existence and convergence of Palais-Smale sequence}.
Our construction is inspired by the classical min-max method.
Notice that $u\equiv0$ is a stable critical point of the functional $I$ in $\overline{\mathcal{C}}$ and our first step is to find an endpoint function $\mathfrak{e}$ satisfying $I(\mathfrak{e})<I(0)=0$. In particular, there exists a positive constant $\sigma$, such that for any homotopy curve $\gamma(s):[0,1]\to\overline{\mathcal{C}}$ satisfying
\begin{equation*}
    \gamma(0;x,y) = 0\mbox{ and }\gamma(1;x,y)=\mathfrak{e}(x,y),
\end{equation*}
it must hold $\displaystyle\max_{s\in[0,1]} I\big(\gamma(s)\big)\geq\frac{\sigma}{2}>0$.

In the min-max method, we construct the Palais-Smale sequence $\{u_{n}\}$ and show the compactness of the sequence. In particular, let $\{\gamma_{n}\}\subseteq\Gamma$ be the minimizing sequence connecting $0$ and $\mathfrak{e}$, and let $u_{n}=\gamma_{n}(s_{n})$ be the corresponding maximizer of $I$. Then one has
\begin{equation*}
\lim_{n\to \infty} I(u_{n}) = \inf_{\gamma\in\mathcal{C},\gamma(0)\equiv0,\gamma(1)\equiv \mathfrak{e}}\max_{s\in[0,1]}I\big(\gamma(s)\big).
\end{equation*}
The weak limit of the Palais-Smale sequence is a saddle point of the functional $I$.

We first verify that $0$ is a stable critical point of $I$ in $\overline{\mathcal{C}}$. Next, we find $\mathfrak{e}$, and construct a pre-Palais-Smale sequence of $\{u_{n}\}$. In order to show that $\{u_{n}\}$ converges subsequentially, it is necessary that $\{\nabla I(u_{n})\}$ converges subsequentially weakly to $0$. Since the space $\overline{\mathcal{C}}$ is not invariant under the gradient flow of $I$ in $H_{0}^1$, it is difficult to get a Palais-Smale sequence via the gradient flow. Our key idea is to deform the homotopy curve using the heat flow instead.

\subsection{Construction of pre-Palais-Smale sequence }
In this subsection,  we first verify that  $u=0$ is a local strictly minimizer of $I$ in $\overline{\mathcal{C}}$. Next, we  find $\mathfrak{e}$ and then construct a pre-Palais-Smale sequence of $\{u_{n}\}$.
\begin{lemma}[Mountain-passes and end points]\label{lem. mountain pass end point}
    Let $I(u)$ be defined as in \eqref{eq. I(u) definition}, with $G(u,y)$ satisfying the assumptions (G1)-(G3). Then, $u=0$ is a stable critical point of $I$ in $\overline{\mathcal{C}}$, and there exists an $\mathfrak{e}\in\mathcal{C}$ with $I(\mathfrak{e})<0$, such that for any continuous homotopy $\gamma(s):[0,1]\to H^{1}_{0}(\Omega)$ with $\gamma(0)\equiv0$ and $\gamma(1)\equiv \mathfrak{e}$, there exists a uniform $\sigma_{0}>0$ such that $\displaystyle\max_{s\in[0,1]}I(\gamma_{s})\geq\sigma_{0}$.
\end{lemma}
\begin{proof}
    \textbf{Step 1. Convexity of $I(\cdot)$ near $0$.} It follows from the assumptions (G1)-(G3) that there exists a $C_{1}>0$ to guarantee
    \begin{equation*}
        G(u,y)\geq-C_{1}|u|^{3}\quad\mbox{for }u\geq0.
    \end{equation*}
    Then, let us fix a small $\delta\in(0,\frac{1}{4C_{1}\mathfrak{C}(3)})$, where $\mathfrak{C}(p)$ is the Sobolev constant with $p=3$ from Lemma~\ref{lem. Sobolev in a channel}. Take an arbitrary $u\in H^{1}_{0}(\Omega)$ with $\|u\|_{H^{1}(\Omega)}=\delta$, then by Lemma~\ref{lem. Sobolev in a channel}, it holds that
    \begin{equation*}
        \int_{\Omega}|u|^{3}dxdy\leq\mathfrak{C}(3)\delta^{3}\leq\frac{\delta^{2}}{4C_{1}}.
    \end{equation*}
    
    Therefore, it holds that
    \begin{equation*}
        I(u)\geq\int_{\Omega}\frac{|\nabla u|^{2}}{2}dxdy-C_{1}\int_{\Omega}|u|^{3}dxdy\geq\frac{\delta^{2}}{2}-C_{1}\mathfrak{C}(3)\delta^{3}\geq\frac{\delta^{2}}{4}.
    \end{equation*}

    \textbf{Step 2. Construction of $\mathfrak{e}\in\mathcal{C}$.} For computational simplicity, we look for
    \begin{equation*}
        \mathfrak{e}(x,y)=a(x)b(y),
    \end{equation*}
    where $a(x)\in C^{\infty}_{0}(\mathbb{R})$ is even and non-increasing in $|x|$, and $b(y)\in C^{\infty}_{0}([-1,1])$.
    
    In fact, let $k$ be sufficiently large and let us set
    \begin{equation*}
        b(y)=k\cdot\exp{\Big(\frac{1}{y^{2}-0.81}\Big)}\chi_{\{|y|<0.9\}}.
    \end{equation*}
    Recalling the assumptions (G1)-(G3) gives
    \begin{equation*}
        G(u,y)\leq C_{1}u^{2}-\frac{1}{C_{1}}u^{3}\quad\mbox{for }u\geq0.
    \end{equation*}
    Then for large $k$, one has
    \begin{equation}\label{eq. I(B) very negative}\begin{aligned}
        \int_{-1}^{1}\Big\{\frac{|\nabla b(y)|^{2}}{2}+G\Big(b(y),y\Big)\Big\}dy\leq&\int_{-1}^{1}\Big\{\frac{|\nabla b(y)|^{2}}{2}+C_{1}b(y)^{2}\Big\}dy-\int_{-1}^{1}\frac{b(y)^{3}}{C_{1}}dy\\
        =&C_{2}k^{2}-c_{3}k^{3}\leq-1.
    \end{aligned}\end{equation}
    
    Let $L$ be large. Assume that  $a(x)=a(|x|)\geq0$ is decreasing in $|x|$ and satisfies
    \begin{equation*}
        a(x)\equiv1\mbox{ for }x\in[-L,L],\quad a(x)\equiv0\mbox{ for }x\in[-L-2,L+2]^{c},\quad|a'(|x|)|\leq1\mbox{ for }x\in\mathbb{R}.
    \end{equation*}
    It is easy to see that one has
    \begin{equation}\label{eq. u_end good}
        \|\mathfrak{e}\|_{H^{1}(\Omega)}\geq2\delta\mbox{ and }I(\mathfrak{e})\leq-1,
    \end{equation}
provided that $L$ is sufficiently large.
    In fact, as $\mathfrak{e}(x,y)=b(y)$ for $x\in[-L,L]$, as long as $L$ is sufficiently large, the first inequality in \eqref{eq. u_end good} is a consequence of the following estimate:
    \begin{equation*}
        \|\mathfrak{e}\|_{H^{1}(\Omega)}^{2}\geq2L\int_{-1}^{1}|\nabla b(y)|^{2}dy=c_{4}L\geq4\delta^{2}.
    \end{equation*}
    The second inequality in \eqref{eq. u_end good} follows from the largeness of $L$ and the negativity property \eqref{eq. I(B) very negative}. In fact, for a fixed $b(y)$, we can find a fixed constant $C_{2}$ such that for $L>100$, the estimate
    \begin{equation*}
        \Big\{\int_{-L-2}^{-L}+\int_{L}^{L+2}\Big\}\int_{-1}^{1}\Big\{\frac{|\nabla \mathfrak{e}|^{2}}{2}+G(\mathfrak{e},y)\Big\}dxdy\leq C_{5}
    \end{equation*}
    holds. As $\mathfrak{e}(x,y)=b(y)$ in $[-L,L]\times[-1,1]$, by choosing a large $L$, it follows from \eqref{eq. I(B) very negative} that one has
    \begin{align*}
        \int_{[-L,L]\times[-1,1]}\Big\{\frac{|\nabla \mathfrak{e}|^{2}}{2}+G(\mathfrak{e},y)\Big\}dxdy=&2L\int_{-1}^{1}\Big\{\frac{|\nabla b(y)|^{2}}{2}+G\Big(b(y),y\Big)\Big\}dy\\
        \leq&-2L\leq-1-2C_{5}.
    \end{align*}
    In summary, one has $I(\mathfrak{e})\leq-1-C_{5}\leq-1$. This proves the second inequality in \eqref{eq. u_end good}.

    \textbf{Step 3. Strict positivity of the maximal energy.} It follows from the construction of $a(x)$ and $b(y)$ that $\mathfrak{e}\in\mathcal{C}$. Using \eqref{eq. u_end good} and the intermediate value theorem, if $\gamma$ is a continuous homotopy in $H^{1}_{0}(\Omega)$ linking $0$ and $\mathfrak{e}$, then one has
    \begin{equation*}
        \|\gamma(s^{*})\|_{H^{1}(\Omega)}=\delta\quad\mbox{for some }s^{*}\in(0,1).
    \end{equation*}
    Therefore, it holds that 
\[
\displaystyle\max_{s\in[0,1]}I\big(\gamma(s)\big)\geq I\big(\gamma(s^{*})\big)\geq\frac{\delta^{2}}{4}.
\]
Hence the proof of Lemma~\ref{lem. mountain pass end point} is completed by setting $\displaystyle\sigma_{0}:=\frac{\delta^{2}}{4}$.
\end{proof}

For the function $\mathfrak{e}$ constructed in Lemma~\ref{lem. mountain pass end point}, define
\begin{equation}\label{eq. define Gamma}
    \Upsilon:=\Big\{\gamma:[0,1]\to\mathcal{C}:\ \gamma(0)\equiv0,\ \gamma(1)\equiv\mathfrak{e},\ \lim_{d\to 0}\sup_{|s-t|\leq d}\|\gamma(s)-\gamma(t)\|_{H^{4}(\Omega)}=0\Big\},
\end{equation}
and
\begin{equation}\label{eq. define Gamma bar}
    \overline{\Upsilon}:=\Big\{\gamma:[0,1]\to\overline{\mathcal{C}}:\ \gamma(0)\equiv0,\ \gamma(1)\equiv\mathfrak{e},\ \lim_{d\to 0}\sup_{|s-t|\leq d}\|\gamma(s)-\gamma(t)\|_{H^{1}(\Omega)}=0\Big\}.
\end{equation}
Then we denote
\begin{equation}\label{eq. minimal of homotopy}  \sigma:=\inf_{\gamma\in\overline{\Upsilon}}\max_{s\in[0,1]}I\big(\gamma(s)\big).
\end{equation}

As $H^{4}\subseteq C^{2,\alpha}$ for all $\alpha\in(0,1)$, any homotopy $\gamma\in\Gamma$ is also continuous in $C^{2,\alpha}(\Omega)$. Since $\overline{\mathcal{C}}\subseteq H^{1}_{0}(\Omega)$ and $H^{4}\subseteq H^{1}$, one has $\sigma\geq\sigma_{0}>0$, where $\sigma_{0}$ is  the constant appeared in Lemma~\ref{lem. mountain pass end point}. Finally, it follows from the inclusion $\mathcal{C}\subseteq\overline{\mathcal{C}}$ that $\displaystyle\inf_{\gamma\in\Upsilon}\max_{s\in[0,1]}I\big(\gamma(s)\big)\geq\sigma$. In fact, the more careful analysis in the following lemma shows that $\displaystyle\inf_{\gamma\in\Upsilon}\max_{s\in[0,1]}I\big(\gamma(s)\big)=\sigma$.
\begin{lemma}\label{lem. rearrange the mountain pass}
  Let $I(u)$ be defined as in \eqref{eq. I(u) definition}, with $G(u,y)$ satisfying the assumptions (G1)-(G3). Let $\Upsilon$ and $\overline{\Upsilon}$ be defined in \eqref{eq. define Gamma} and \eqref{eq. define Gamma bar}, respectively. Then
    \begin{equation*}
\inf_{\gamma\in\Upsilon}\max_{s\in[0,1]}I\big(\gamma(s)\big)=\inf_{\gamma\in\overline{\Upsilon}}\max_{s\in[0,1]}I\big(\gamma(s)\big)=\sigma.
    \end{equation*}
\end{lemma}
\begin{proof}
    For each $n\in\mathbb{N}$, there exists a homotopy $\widetilde{\gamma}\in\overline{\Upsilon}$ such that
    \begin{equation*}
        \max_{s\in[0,1]}I\big(\widetilde{\gamma}(s)\big)\leq\sigma+\frac{1}{n^{3}}.
    \end{equation*}
    For such a fixed homotopy, denote $\displaystyle M=\max_{s\in[0,1]}\|\widetilde{\gamma}(s)\|_{L^{2}(\Omega)}+1$.
    
    \textbf{Step 1. Construction of $\gamma$ at sample points.} Let $N\in\mathbb{N}$ and let $s_{i}=\frac{i}{N}$ for $i\in\{0,1,2,\cdots,N\}$. Here, $N$ is chosen so that
    \begin{equation}\label{eq. how to choose dense sample points}
        \|\widetilde{\gamma}(s_{i+1})-\widetilde{\gamma}(s_{i})\|_{H^{1}(\Omega)}\leq\frac{1}{M\cdot n^{3}},\quad\mbox{for all }0\leq i<N.
    \end{equation}
 Now we construct a homotopy $\gamma\in\Gamma$ such that $\displaystyle\max_{s\in[0,1]}I\big(\gamma(s)\big)\leq\sigma+\frac{1}{n}$.
    
    First, we set $\gamma(s_{0})\equiv0$ and $\gamma(s_{N})\equiv\mathfrak{e}$. For  $0<i<N$, $\widetilde{\gamma}(s_{i})$ can be approximated in $H^{1}(\Omega)$ by a sequence of functions in $\mathcal{C}$. Therefore, one can find $\gamma(s_{i})\in\mathcal{C}$ such that
    \begin{equation}\label{eq. H1 approx of nonsmooth homotopy}
        \|\gamma(s_{i})-\widetilde{\gamma}(s_{i})\|_{H^{1}(\Omega)}\leq\frac{1}{M\cdot n^{3}}\quad\mbox{and }\max_{0\leq i\leq N}I\big(\gamma(s_{i})\big)\leq\sigma+\frac{1}{M\cdot n^{2}}.
    \end{equation}
    
    \textbf{Step 2. Construction of $\gamma$ at other points.} For $s\in(s_{i},s_{i+1})$, define 
    \begin{equation}\label{def_lambda_gamma}
    \lambda(s)=N\cdot(s-s_{i})  \quad\text{and}\quad   \gamma(s)=(1-\lambda(s))\cdot\gamma(s_{i})+\lambda(s)\cdot\gamma(s_{i+1}).
    \end{equation}
    Clearly, $\gamma$ is a broken line connecting the sample points $\gamma(s_{i})\in\mathcal{C}$, so $\gamma$ must be a homotopy in the $H^{4}$ sense connecting $0$ and $\mathfrak{e}$. Since $\mathcal{C}$ is a convex set, one has $\gamma(s)\in\mathcal{C}$ for all $s\in[0,1]$. In other words, this curve $\gamma$ belongs to $\Upsilon$.

    \textbf{Step 3. Energy estimate.} Assume that $s\in[s_{i},s_{i+1}]$. Decompose
    \begin{equation*}
        I(u)=\int_{\Omega}\frac{|\nabla u|^{2}}{2}dxdy+\int_{\Omega}G(u,y)dxdy=:I_{D}(u)+I_{P}(u).
    \end{equation*}
    Since $I_{D}(\cdot)$ is convex, one has
    \begin{equation}\label{eq. IA convex, estimate good}
        I_{D}\big(\gamma(s)\big)\leq(1-\lambda)I_{D}\big(\gamma(s_{i})\big)+\lambda I_{D}\big(\gamma(s_{i+1})\big),
    \end{equation}
    where $\lambda=\lambda(s)$ is defined in \eqref{def_lambda_gamma}.
    On the other hand, let us set $\alpha=\gamma(s_{i})$, $\beta=\gamma(s_{i+1})$. From the assumption (G3), one has
    \begin{equation*}
        G\Big((1-\lambda)\alpha+\lambda\beta,y\Big)-(1-\lambda)G(\alpha,y)-\lambda G(\beta,y)\leq C(1+|\alpha|+|\beta|)|\alpha-\beta|^{2}.
    \end{equation*}
    Applying the Cauchy-Schwarz inequality gives
    \begin{equation}\label{eq. concave deficit}\begin{aligned}
        &\, -(1-\lambda)I_{P}\big(\gamma(s_{i})\big)-\lambda I_{P}\big(\gamma(s_{i+1})\big)+I_{P}\big(\gamma(s)\big)\\
        \leq&\, C\int_{\Omega}\Big(1+\gamma(s_{i})+\gamma(s_{i+1})\Big)\cdot|\gamma(s_{i})-\gamma(s_{i+1})|^{2}dxdy\\
        \leq& \, C\big(\|\gamma(s_{i})\|_{L^{2}(\Omega)}+\|\gamma(s_{i+1})\|_{L^{2}(\Omega)}\big)\cdot\|\gamma(s_{i})-\gamma(s_{i+1})\|_{L^{4}(\Omega)}^{2}\\
        &\, +C\|\gamma(s_{i})-\gamma(s_{i+1})\|_{L^{2}(\Omega)}^{2},
    \end{aligned}\end{equation}
 where $\lambda=\lambda(s)$ is defined in \eqref{def_lambda_gamma}.
 
    By \eqref{eq. how to choose dense sample points} and \eqref{eq. H1 approx of nonsmooth homotopy}, one has
    \begin{equation*}
        \|\gamma(s_{i})-\gamma(s_{i+1})\|_{H^{1}(\Omega)}\leq\frac{1}{M\cdot n^{2}}.
    \end{equation*}
    Hence it follows from Lemma~\ref{lem. Sobolev in a channel} that
    \begin{equation*}
        \|\gamma(s_{i})-\gamma(s_{i+1})\|_{L^{2}(\Omega)}\leq\frac{\mathfrak{C}(2)}{M\cdot n^{2}}\quad\mbox{and }\|\gamma(s_{i})-\gamma(s_{i+1})\|_{L^{4}(\Omega)}\leq\frac{\mathfrak{C}(4)}{M\cdot n^{2}}.
    \end{equation*}
    Next, by \eqref{eq. H1 approx of nonsmooth homotopy} and Lemma~\ref{lem. Sobolev in a channel}, one has
    \begin{equation*}
        \|\gamma(s_{i})\|_{L^{2}(\Omega)}+\|\gamma(s_{i+1})\|_{L^{2}(\Omega)}\leq2\max_{s\in[0,1]}\|\widetilde{\gamma}(s)\|_{L^{2}(\Omega)}+2\sqrt{\mathfrak{C}(2)}\cdot\frac{1}{n^{3}}\leq2M.
    \end{equation*}
    It then follows from \eqref{eq. concave deficit} that
    \begin{equation*}
        -(1-\lambda)I_{P}\big(\gamma(s_{i})\big)-\lambda I_{P}\big(\gamma(s_{i+1})\big)+I_{P}\big(\gamma(s)\big)\leq\frac{1}{n^{2}}.
    \end{equation*}
    This, together with \eqref{eq. IA convex, estimate good}, yields that
    \begin{equation*}
       I\big(\gamma(s)\big)\leq(1-\lambda)I\big(\gamma(s_{i})\big)+\lambda I\big(\gamma(s_{i+1})\big)+\frac{1}{n^{2}}\leq\sigma+\frac{1}{M\cdot n^{2}}+\frac{1}{n^{2}}\leq\sigma+\frac{1}{n},
    \end{equation*}
    where we have used \eqref{eq. H1 approx of nonsmooth homotopy} again to control $I\big(\gamma(s_{i})\big)$ and $I\big(\gamma(s_{i+1})\big)$. This finishes the proof of Lemma~\ref{lem. rearrange the mountain pass}.
\end{proof}

\subsection{Heat flow algorithm and construction of  a Palais-Smale sequence.}
In this subsection, we construct a Palais-Smale sequence $\{u_{n}\}\subseteq\overline{\mathcal{C}}$. Notice that $L^{2}(\Omega)$ space naturally embeds in the space $H^{-1}(\Omega)$ by the Riesz representation theorem and the Sobolev inequality (see Lemma~\ref{lem. Sobolev in a channel} with $p=2$), it suffices to require
\begin{equation}\label{eq. PS requirement, written in L2}
    u_{n}\in\overline{\mathcal{C}},\quad|I(u_{n})-\sigma|\leq\frac{1}{n},\quad\|\nabla I(u_{n})\|_{L^{2}(\Omega)}\leq\frac{1}{n}.
\end{equation}
Here, $\sigma$ is the min-max value in \eqref{eq. minimal of homotopy}.

The main goal of this subsection is to prove the following lemma.
\begin{lemma}\label{lem. key to prove flexibility, finding the PS sequence}
    There exists a sequence $\{u_{n}\}$ such that \eqref{eq. PS requirement, written in L2} holds.
\end{lemma}
In fact, Lemma~\ref{lem. key to prove flexibility, finding the PS sequence} immediately implies the first part of Proposition~\ref{prop. Existence and convergence of Palais-Smale sequence}.
\begin{proof}[Proof for Part (1) of Proposition~\ref{prop. Existence and convergence of Palais-Smale sequence}]
    Assuming that Lemma~\ref{lem. key to prove flexibility, finding the PS sequence} holds, there exists a sequence $\{u_{n}\}$ satisfying \eqref{eq. PS requirement, written in L2}. Using \eqref{eq. L2 way to understand nabla I} and Lemma~\ref{lem. Sobolev in a channel} with $p=2$, it follows from the Riesz representation theorem that the sequence $\{u_{n}\}$ also satisfies $\|\nabla I(u_{n})\|_{H^{-1}(\Omega)}\to0$. Hence, the first part of Proposition~\ref{prop. Existence and convergence of Palais-Smale sequence} is proved.
\end{proof}

To find the function $u_{n}$ (where $n$ is sufficiently large) required in Lemma~\ref{lem. key to prove flexibility, finding the PS sequence}, we let $\gamma\in\Gamma$ be a homotopy connecting $0$ and $\mathfrak{e}$, which satisfies the inequality
\begin{equation}\label{eq. homotopy before heat flow deformation}
    \max_{s\in[0,1]}I\big(\gamma(s)\big)\leq\sigma+\frac{1}{n^{4}}.
\end{equation}
The existence of such a homotopy is guaranteed by Lemma~\ref{lem. rearrange the mountain pass}. To guarantee the last inequality of \eqref{eq. PS requirement, written in L2}, we use the following heat flow algorithm to deform the homotopy $\gamma$.

\begin{definition}[Heat flow algorithm]\label{def. Heat flow algorithm}
    Let $n$ be sufficiently large, and let $\gamma\in\Gamma$ be a homotopy that satisfies \eqref{eq. homotopy before heat flow deformation}. For each $s\in[0,1]$, let $v$ be the solution to the problem:
    \begin{equation*}
            \left\{\begin{aligned}
                &\frac{\partial v}{\partial t}=\Delta v-G'(v,y)=:\mathcal{L}v,\\
                &v(t,\cdot)\in H^{1}_{0}(\Omega),\\
                &v(0,\cdot)=\gamma(s;\cdot).
            \end{aligned}\right.
        \end{equation*}
    Denote $e^{t\mathcal{L}}\gamma(s):=v(t,\cdot)$. The stop time of this algorithm, denoted by $\overline{t}$, is the supremum of all possible $T\geq0$ satisfying all three requirements below.
    \begin{itemize}
        \item[(K1)] $\displaystyle I\big(\gamma(s)\big)-I\big(e^{t\mathcal{L}}\gamma(s)\big)<\frac{1}{n^{2}}$ for all $t\in[0,T)$.
        \item[(K2)] $\displaystyle-\frac{d}{dt}I\big(e^{t\mathcal{L}}\gamma(s)\big)>\frac{1}{n^{2}}$ for all $t\in[0,T)$.
        \item[(K3)] $\displaystyle\sup_{t\in[0,T)}\|e^{t\mathcal{L}}\gamma(s)\|_{L^{\infty}(\Omega)}$ is finite (this value is allowed to depend on $T$).
    \end{itemize}
    Here, the set $[0,0)$ is viewed as the empty set.
\end{definition}
\begin{remark}\label{rmk. computations laws in heat flow algorithm}
    Using the notation in Definition~\ref{def. Heat flow algorithm} yields
    \begin{equation*}
        \|\nabla I(v)\|_{L^{2}(\Omega)}=\|\mathcal{L}v\|_{L^{2}(\Omega)}\quad\mbox{and }-\frac{d}{dt}I\big(e^{t\mathcal{L}}\gamma(s)\big)=\|\mathcal{L}e^{t\mathcal{L}}\gamma(s)\|_{L^{2}(\Omega)}^{2}.
    \end{equation*}
\end{remark}
\begin{remark}
    Since $\gamma(s)\in\mathcal{C}$, in particular, $\gamma(s)=\Delta\gamma(s,\cdot)=0$ on $\partial\Omega=\mathbb{R}\times\{\pm1\}$, we infer from Lemma~\ref{lem. appendix well-posedness} that the heat flow algorithm in Definition~\ref{def. Heat flow algorithm} exists at least for a short time, with $e^{t\mathcal{L}}\gamma(s)\in\overline{\mathcal{C}}$. Furthermore, it follows from the proof of Lemma~\ref{lem. appendix well-posedness} and the finiteness of $\displaystyle\sup_{t\in[0,T]}\|e^{t\mathcal{L}}\gamma(s)\|_{L^{\infty}(\Omega)}$ that $e^{t\mathcal{L}}\gamma(s)$ is smooth in $\Omega$ for all $t\in[0,\overline{t}]$.
\end{remark}
\begin{remark}\label{rmk. stop reason trichotomy}
    Corresponding to (K1)-(K3) in Definition~\ref{def. Heat flow algorithm}, at $\overline{t}$, the reason that the heat flow algorithm cannot proceed must be one of the following.
    \begin{itemize}
        \item[(S1)] Total energy decay is too big, i.e., $\displaystyle I\big(\gamma(s)\big)-I\big(e^{\overline{t}\mathcal{L}}\gamma(s)\big)\geq\frac{1}{n^{2}}$.
        \item[(S2)] Energy decay speed is too small, i.e., $\displaystyle\lim_{\epsilon\to0+}\inf_{t\in[\overline{t},\overline{t}+\epsilon]}\Big\{-\frac{d}{dt}I\big(e^{t\mathcal{L}}\gamma(s)\big)\Big\}\leq\frac{1}{n^{2}}$.
        \item[(S3)] The heat flow blows up. In particular, $\displaystyle\sup_{[0,\overline{t}]}\|e^{t\mathcal{L}}\gamma(s)\|_{L^{\infty}(\Omega)}=\infty$.
    \end{itemize}
\end{remark}

We have the criterion (K3) in Definition~\ref{def. Heat flow algorithm} to ensure that the heat flow $v=e^{t\mathcal{L}}\gamma(s)$ is well-defined. In fact, the following lemma indicates that the criterion (K3) is indeed redundant as long as $n$ is sufficiently large.
\begin{lemma}\label{lem. no blow up in the algorithm}
    Let $n$ be sufficiently large (independent of the curve $\gamma\in\Gamma$ and $s\in[0,1]$). For any $\gamma\in\Gamma$ and $s\in[0,1]$, the heat flow algorithm in Definition~\ref{def. Heat flow algorithm} does not blow up before the stop time $\overline{t}$, i.e., $\displaystyle\sup_{t\in[0,\overline{t}]}\|e^{t\mathcal{L}}\gamma(s)\|_{L^{\infty}(\Omega)}<\infty$.
\end{lemma}
\begin{proof}
    Suppose that the algorithm stops at $\overline{t}$ because the heat flow blows up, then criteria (K1) and (K2) must hold for all $t\in[0,\overline{t}]$, i.e.,
    \begin{equation*}
        -\frac{d}{dt}I\big(e^{t\mathcal{L}}\gamma(s)\big)\geq\frac{1}{n^{2}}\mbox{ and }I\big(\gamma(s)\big)-I\big(e^{t\mathcal{L}}\gamma(s)\big)\leq\frac{1}{n^{2}}\quad\mbox{for all }t\in[0,\overline{t}].
    \end{equation*}
    This immediately implies $\overline{t}\leq1$. Furthermore, it follows from Remark~\ref{rmk. computations laws in heat flow algorithm} that
    \begin{equation*}
        \min_{t\in[0,\overline{t}]}\|\mathcal{L}e^{t\mathcal{L}}\gamma_{s}\|_{L^{2}(\Omega)}=\min_{t\in[0,\overline{t}]}\sqrt{-\frac{d}{dt}I(e^{t\mathcal{L}}\gamma_{s})}\geq\frac{1}{n}.
    \end{equation*}
Consequently, for every $\tau\in(0,\overline{t})$
    \begin{align*}
    \frac{1}{n^{2}}\geq&\, I(\gamma_{s})-I(e^{\tau\mathcal{L}}\gamma_{s})=\int_{0}^{\tau}-\frac{d}{dt}I\big(e^{t\mathcal{L}}\gamma(s)\big)dt=\int_{0}^{\tau}\|\mathcal{L}e^{t\mathcal{L}}\gamma(s)\|_{L^{2}(\Omega)}^{2}dt\\
\geq&\,\inf_{t\in[0,\overline{t}]}\|\mathcal{L}e^{t\mathcal{L}}\gamma(s)\|_{L^{2}(\Omega)}\int_{0}^{\tau}\|\mathcal{L}e^{t\mathcal{L}}\gamma(s)\|_{L^{2}(\Omega)}dt\\
    \geq&\,\frac{1}{n}\int_{0}^{\tau}\Big\|\frac{\partial e^{t\mathcal{L}}\gamma(s)}{\partial t}\Big\|_{L^{2}(\Omega)}dt\geq\frac{1}{n}\cdot\|e^{\tau\mathcal{L}}\gamma(s)-\gamma(s)\|_{L^{2}(\Omega)}.
\end{align*}
Therefore, we have
\begin{equation}\label{eq. L2 small difference with the initial data}
    \|e^{\tau\mathcal{L}}\gamma(s)-\gamma(s)\|_{L^{2}(\Omega)}\leq\frac{1}{n}\quad\mbox{for all }\tau\in[0,\overline{t}].
\end{equation}
It follows from Lemma~\ref{lem. epsilon regularity} that the difference $e^{\tau\mathcal{L}}\gamma(s)-\gamma(s)$ is bounded in $L^{\infty}(\Omega)$. To see this, let $w(t,\cdot):=e^{t\mathcal{L}}\gamma(s;\cdot)-\gamma(s;\cdot)$. The straightforward computations yield
\begin{equation*}
    (\partial_{t}-\Delta)w=\Delta\gamma-G'(w+\gamma,y).
\end{equation*}
Using the assumptions (G1)-(G3), one has
\begin{equation*}
    \left\{\begin{aligned}
        &\big|(\partial_{t}-\Delta)w\big|\leq C_{0}w^{2}+\alpha(x,y)|w|+\beta(x,y),\\
        &w(0,\cdot)\equiv0,\quad w(t,\cdot)\in H^{1}_{0}(\Omega),
    \end{aligned}\right.
\end{equation*}
where
\begin{equation*}
    \alpha(\cdot)=C|\gamma(s;\cdot)|\mbox{ and }\beta(\cdot)=C\Big(|\Delta\gamma(s;\cdot)|+\gamma^{2}(s;\cdot)+1\Big).
\end{equation*}
As $\overline{t}\leq1$, and by \eqref{eq. L2 small difference with the initial data} one has
\begin{equation*}
    \|w(\tau,\cdot)\|_{L^{2}(\Omega)}=\displaystyle\Big\|e^{\tau\mathcal{L}}\gamma(s)-\gamma(s)\Big\|_{L^{2}(\Omega)}\leq\frac{1}{n}\quad\mbox{for all }\tau\leq\overline{t}.
\end{equation*}
For sufficiently large $n$, one can apply Lemma~\ref{lem. epsilon regularity} to $w_{\pm}(t,\cdot)$, and obtain that
\begin{equation*}
    \|w(t,\cdot)\|_{L^{\infty}(\Omega)}\leq\sup_{\Omega}\Big(|\alpha(x,y)|+|\beta(x,y)|\Big)<\infty,\quad\mbox{for all }t\leq\overline{t}.
\end{equation*}
Given the $L^{\infty}$ bound of $w$, we then have the uniform (independent of $\tau$) bound of $L^{\infty}$ norm of $e^{t\mathcal{L}}\gamma(s)$ for $t\in[0,\tau]$ for every $\tau\leq\overline{t}$. With this information, we can repeat the proof of Lemma~\ref{lem. appendix well-posedness}, and show the uniform $H^{2}(\Omega)$ estimate for $e^{t\mathcal{L}}\gamma(s)$ for $t\in[0,\overline{t}]$. It follows from the bootstrap argument that $w(t,\cdot)\in C^{k}(\Omega)$ for $t\in[0,\overline{t}]$.

If we choose $e^{\overline{t}\mathcal{L}}\gamma(s)$ as the initial data and apply Lemma~\ref{lem. appendix well-posedness} again, then $e^{t\mathcal{L}}\gamma(s)$ remains classical and smooth even slightly beyond $\overline{t}$. This contradicts the assumption that the heat flow blows up. Hence the proof of Lemma~\ref{lem. no blow up in the algorithm} is completed.
\end{proof}

Next, we let $n$ be sufficiently large, so that the blow-up of heat equation deos not occur. the following lemma is essential in proving Lemma~\ref{lem. key to prove flexibility, finding the PS sequence}.
\begin{lemma}\label{lem. it is possible that (K2) is the stopping reason}
    Let $n$ be sufficiently large, and let $\gamma\in\Upsilon$ be a homotopy curve that satisfies \eqref{eq. homotopy before heat flow deformation}. Then there exists an $s\in[0,1]$ such that
    \begin{equation}\label{eq. which element on the curve do we perform heat flow?}
        |I\big(\gamma(s)\big)-\sigma|\leq\frac{1}{n^{4}},
    \end{equation}
    and the heat flow $e^{t\mathcal{L}}\gamma(s)$ stops at $\overline{t}$ with $\displaystyle\lim_{\epsilon\to0+}\inf_{t\in[\overline{t},\overline{t}+\epsilon]}\Big\{-\frac{d}{dt}I\big(e^{t\mathcal{L}}\gamma(s)\big)\Big\}\leq\frac{1}{n^{2}}$, i.e. (S2) is the reason for the heat flow to stop. 
\end{lemma}
Assume that Lemma~\ref{lem. it is possible that (K2) is the stopping reason} holds, we first prove Lemma~\ref{lem. key to prove flexibility, finding the PS sequence}.
\begin{proof}[Proof of Lemma~\ref{lem. key to prove flexibility, finding the PS sequence}]
    Assume that $s$ is the one appeared in Lemma~\ref{lem. it is possible that (K2) is the stopping reason}, and let $\overline{t}$ be the corresponding stop time of the heat flow, then
    \begin{equation}\label{eq. PS requirement, check 1}
        \lim_{\epsilon\to0+}\inf_{t\in[\overline{t},\overline{t}+\epsilon]}\|\nabla I\big(e^{t\mathcal{L}}\gamma(s)\big)\|_{L^{2}(\Omega)}\leq\sqrt{-\frac{d}{dt}I\big(e^{t\mathcal{L}}\gamma(s)\big)}\leq\frac{1}{n}.
    \end{equation}
    As $n\gg1$, the heat flow remains smooth for $t\in[0,\overline{t}+\epsilon]$.
    
    Since the reason for heat flow to stop in Lemma~\ref{lem. it is possible that (K2) is the stopping reason} is (S2) instead of (S1), we have:
    \begin{equation*}
        \lim_{t\to\overline{t}+}\Big\{I\big(\gamma(s)\big)-I\big(e^{t\mathcal{L}}\gamma(s)\big)\Big\}=I\big(\gamma(s)\big)-I\big(e^{\overline{t}\mathcal{L}}\gamma(s)\big)\leq\frac{1}{n^{2}}
    \end{equation*}
    Combining this estimate with the assumption \eqref{eq. which element on the curve do we perform heat flow?} yields
    \begin{equation}\label{eq. PS requirement, check 3}
        \sigma-\frac{1}{n^{4}}-\frac{1}{n^{2}}\leq\lim_{t\to\overline{t}+}I\big(e^{t\mathcal{L}}\gamma(s)\big)\leq\sigma+\frac{1}{n^{4}},
    \end{equation}
    where the monotone decreasing property of $I(e^{t\mathcal{L}}\gamma(s)$ for $t\in (0,\bar{t}]$ has been used. 
    
    Finally, it follows from Part (3) of Lemma~\ref{lem. appendix well-posedness} that $e^{t\mathcal{L}}\gamma(s)\in\overline{\mathcal{C}}$. This, together with \eqref{eq. PS requirement, check 1} and \eqref{eq. PS requirement, check 3}, implies that $u_{n}:=e^{t\mathcal{L}}\gamma(s)$ satisfies \eqref{eq. PS requirement, written in L2} for some $t\in[\overline{t},\overline{t}+\epsilon]$.
\end{proof}
The rest of this section is devoted to the proof of Lemma~\ref{lem. it is possible that (K2) is the stopping reason}.
\begin{proof}[Proof of Lemma~\ref{lem. it is possible that (K2) is the stopping reason}]
    We argue by contradiction. Suppose that there does not exist an $s\in[0,1]$ satisfying the requirements in Lemma~\ref{lem. it is possible that (K2) is the stopping reason}.
    
    \textbf{Step 1.} F all $s\in[0,1]$, let $\tau(s)$ be the stop time $\overline{t}$ of the heat flow algorithm $e^{t\mathcal{L}}\gamma(s)$ (see Definition~\ref{def. Heat flow algorithm}) for each $\gamma(s)$. For sufficiently large $n$, it follows from Lemma~\ref{lem. no blow up in the algorithm} that the heat flow $e^{t\mathcal{L}}\gamma(s)$ does not blow up before the time $\tau(s)$. Furthermore, it follows from Part (2) of Lemma~\ref{lem. appendix well-posedness} that for each $s\in[0,1]$, the heat flow is $H^{1}(\Omega)\cap L^{\infty}(\Omega)$-stable with respect to the initial data. In particular, the function $\tau(s)$ must be lower semi-continuous.

    Denote
    \begin{equation*}
        \mathcal{S}:=\Big\{s\in[0,1]\Big|\,|I\big(\gamma(s)\big)-\sigma|\leq\frac{1}{n^{4}}\Big\}.
    \end{equation*}
    Clearly, $\mathcal{S}$ is a compact subset of $[0,1]$, and neither $0$ nor $1$ belongs to $\mathcal{S}$.
    If the contradiction assumption holds, then it follows from Remark~\ref{rmk. stop reason trichotomy} that
    \begin{equation*}
        I\big(\gamma(s)\big)-I\big(e^{\tau(s)\mathcal{L}}\gamma(s)\big)\geq\frac{1}{n^{2}},\quad\mbox{for all }s\in\mathcal{S}.
    \end{equation*}

    For each $s^{*}\in\mathcal{S}$, define
    \begin{equation*}
        \tau_1(s^{*}):=\sup\Big\{t>0\Big|\ I\big(\gamma(s^{*})\big)-I\big(e^{t\mathcal{L}}\gamma(s^{*})\big)\leq\frac{1}{n^{3}}\Big\}.
    \end{equation*}
    Clearly, $\tau_1(s^*)< \tau(s^*)$. Furthermore, by the semi-lower continuity of $\tau$ mentioned above, there exists an $r(s^{*})>0$, such that for all $s$ satisfying
    \begin{equation*}
        s\in\mathcal{N}_{s^{*}}:=\Big(s^{*}-r(s^{*}),s^{*}+r(s^{*})\Big),
    \end{equation*}
    one has
    \begin{equation}\label{eq. tau_1 stable property}
        \tau(s)\geq\tau_{1}(s^{*}),\quad\mbox{and }I\big(\gamma(s)\big)-I\big(e^{\tau_{1}(s^{*})\mathcal{L}}\gamma(s)\big)\geq\frac{1}{2n^{3}}.
    \end{equation}
    
    \textbf{Step 2.} In the neighborhood $\mathcal{N}_{s^{*}}$, we define a new stop time $\mathcal{T}_{s^{*}}(s)$ for the heat flow algorithm starting from $\gamma(s)$, which is always less than the actual stop time $\tau(s)$. In fact, define a continuous cut-off function
    \begin{equation*}
        \rho(s)=\left\{\begin{aligned}
            &1,&\mbox{if }&I\big(\gamma(s)\big)\geq\sigma,\\
            &1-n^{4}\cdot\Big[\sigma-I\big(\gamma(s)\big)\Big],&\mbox{if }&\sigma-\frac{1}{n^{4}}\leq I\big(\gamma(s)\big)\leq\sigma,\\
            &0,&\mbox{if }&I\big(\gamma(s)\big)\leq\sigma-\frac{1}{n^{4}}.
        \end{aligned}\right.
    \end{equation*}
    For $s\in\mathcal{N}_{s^{*}}$, let
    \begin{equation*}
        \mathcal{T}_{s^{*}}(s):=\tau_{1}(s^{*})\cdot\rho(s).
    \end{equation*}
    It follows from \eqref{eq. tau_1 stable property} that
    \begin{equation}\label{eq. T local stop time}
        \mathcal{T}_{s^{*}}(s)\leq\tau_{1}(s^{*})\leq\tau(s)\quad\mbox{for all }s\in\mathcal{N}_{s^{*}}.
    \end{equation}
    More importantly, for each $\gamma(s)$ with $s\in\mathcal{N}_{s^{*}}$, at $t=\mathcal{T}_{s^{*}}(s)$ one has
    \begin{equation}\label{eq. any local stop time decreases I() a lot}
        I\Big(e^{\mathcal{T}_{s^{*}}(s)\mathcal{L}}\gamma(s)\Big)<\sigma.
    \end{equation}
    In fact, recall that $I\big(e^{t\mathcal{L}}\gamma(s)\big)$ is decreasing in $t$ (see Remark~\ref{rmk. computations laws in heat flow algorithm}), so it suffices to consider those $s\in\mathcal{N}_{s^{*}}$ satisfying $I\big(\gamma(s)\big)\geq\sigma$. For those $s$, we deduce from \eqref{eq. tau_1 stable property} that:
    \begin{equation*}
        I\Big(e^{\mathcal{T}_{s^{*}}(s)\mathcal{L}}\gamma(s)\Big)=I\Big(e^{\tau_{1}(s^{*})\mathcal{L}}\gamma(s)\Big)\leq I\big(\gamma(s)\big)-\frac{1}{2n^{3}}\leq\sigma+\frac{1}{n^{4}}-\frac{1}{2n^{3}}<\sigma.
    \end{equation*}

    Notice that $\mathcal{S}$ is a compact set excluding the endpoints of the homotopy curve, and that $\{\mathcal{N}_{s^{*}}\}$ forms an open cover of $\mathcal{S}$, we can apply the Heine-Borel theorem, and extract a finite sub-cover of $\{\mathcal{N}_{s^{*}}\}$, denoted by $\{\mathcal{N}_{s_{i}^{*}}\}_{i=1}^{m}$. For $i\in\{1,\cdots,m\}$, write:
    \begin{equation*}
        \mathcal{N}_{i}:=\mathcal{N}_{s_{i}^{*}},\quad\mathcal{T}_{i}(s):=\mathcal{T}_{s_{i}^{*}}(s).
    \end{equation*}
    These sets together satisfy $\mathcal{N}_{1}\cup\cdots\cup\mathcal{N}_{m}\supseteq\mathcal{S}$. We additionally define
    \begin{equation*}
        \mathcal{N}_{0}:=[0,1]\setminus\mathcal{S},\quad\mbox{and }\mathcal{T}_{0}(s):=0\mbox{ for }s\in\mathcal{N}_{0}.
    \end{equation*}
    
    \textbf{Step 3.} Let $\{\eta_{i}\}_{i=0}^{m}$ be a (nonnegative) smooth partition of unity corresponding to the cover $\{\mathcal{N}_{i}\}_{i=0}^{m}$. We define a new stop time on the whole interval $[0,1]$ as
    \begin{equation*}
        \mathcal{T}(s):=\sum_{i=0}^{m}\eta_{i}(s)\mathcal{T}_{i}(s).
    \end{equation*}
    Clearly, $\mathcal{T}(s)$ is a continuous function with respect to $s$. Furthermore, it follows from \eqref{eq. T local stop time} that
    \begin{equation*}
        \mathcal{T}(s)\leq\tau(s)\mbox{ in }[0,1].
    \end{equation*}
    Since $\displaystyle I\big(\gamma(s)\big)\leq\sigma+\frac{1}{n^{4}}$ for $s\in [0,1]$, 
    for $s\notin\mathcal{S}$, it holds that $\displaystyle I\big(\gamma(s)\big)\leq\sigma-\frac{1}{n^{4}}$. Hence, $\mathcal{T}_{i}(s)=0$ for all $0\leq i\leq m$, and thus $\mathcal{T}(s)=0$. In particular, $\mathcal{T}(0)=\mathcal{T}(1)=0$. Next, define a new curve:
    \begin{equation*}
        \widetilde{\gamma}(s):=e^{\mathcal{T}(s)\mathcal{L}}\gamma(s).
    \end{equation*}
    As $\mathcal{T}(0)=\mathcal{T}(1)=0$, we still have $\widetilde{\gamma}(0)=0$ and $\widetilde{\gamma}(1)=\mathfrak{e}$. Moreover, as $\mathcal{T}(s)$ is a continuous function such that $e^{t\mathcal{L}}\gamma(s)$ does not blow up for $t\leq\mathcal{T}(s)$, it then follows from Part (2) of Lemma~\ref{lem. appendix well-posedness} that $\widetilde{\gamma}(s)$ is a continuous curve in $H^{1}(\Omega)$ with respect to $s$. Therefore, $\widetilde{\gamma}\in\overline{\Gamma}$.
    
    For every $s\in[0,1]$ satisfying $I\big(\gamma(s)\big)\geq\sigma$, it must hold that  $s\notin\mathcal{N}_{0}$, so the partition function satisfies $\eta_{0}=0$. This implies
    \begin{equation*}
        \mathcal{T}(s)=\sum_{i=1}^{m}\eta_{i}(s)\mathcal{T}_{i}(s)\mbox{ and }\sum_{i=1}^{m}\eta_{i}(s)=1,\quad\mbox{for all }s\in[0,1]\cap\Big\{I\big(\gamma(s)\big)\geq\sigma\Big\}.
    \end{equation*}Therefore, there exists an $i=i(s)\in\{1,\cdots,m\}$ with 
    $\eta_{i}(s)>0$ such that
    \begin{equation*}
        s\in\mathcal{N}_{i}\mbox{ and }\mathcal{T}(s)\geq\mathcal{T}_{i}(s).
    \end{equation*}
    Hence, it follows from \eqref{eq. any local stop time decreases I() a lot} that for every $s\in[0,1]$ with $I\big(\gamma(s)\big)\geq\sigma$,
    \begin{equation*}
        I\Big(e^{\mathcal{T}(s)\mathcal{L}}\gamma(s)\Big)\leq I\Big(e^{\mathcal{T}_{i}(s)\mathcal{L}}\gamma(s)\Big)<\sigma,\quad\mbox{for some }i=i(s)\in\{1,\cdots,m\}.
    \end{equation*}
    In conclusion, the new homotopy curve $\widetilde{\gamma}$ satisfies
    \begin{equation*}
        \widetilde{\gamma}\in\overline{\Gamma}\quad\mbox{and }\max_{s\in[0,1]}I\big(\widetilde{\gamma}(s)\big)<\sigma.
    \end{equation*}
    This contradicts the definition of $\sigma$ in \eqref{eq. minimal of homotopy}. The proof of Lemma~\ref{lem. it is possible that (K2) is the stopping reason} is completed.
\end{proof}
\subsection{Strong \texorpdfstring{$H^{1}$}{Lg}-convergence}
In this subsection, we prove the following refined Palais-Smale condition. We show that the functional $I(u)$, when restricted to well-arranged functions $u$, satisfies the Palais-Smale condition.
\begin{lemma}[Refined Palais-Smale condition]\label{lem. Palais-Smale}
   Let $I(u)$ be defined as in \eqref{eq. I(u) definition}, with $G(u,y)$ satisfying the assumptions (G1)-(G3). Let $\{u_{n}\}\subseteq\overline{\mathcal{C}}$ be a sequence such that $\{I(u_{n})\}$ is uniformly bounded and that $\nabla I(u_{n})\to0$ in $H^{-1}(\Omega)$. Then there exists a subsequence of $\{u_{n}\}$ which strongly converges to some $u_{\infty}\in H^{1}_{0}(\Omega)$ in $H^{1}(\Omega)$.
\end{lemma}
\begin{proof}
    For simplicity, assume that there exists a constant $C_1>0$ such that
    \begin{equation}\label{eq. a PS sequence already given}
        I(u_{n})\leq C_{1}\mbox{ and }\|\nabla I(u_{n})\|_{H^{-1}(\Omega)}\leq\frac{1}{n}\quad \text{for all}\,\, n\in \mathbb{N}.
    \end{equation}
    Moreover, it follows from the assumptions (G1)-(G3) that
    \begin{equation}\label{eq. Legendre transform of G(u,y)}
        G(u,y)\geq-Cu^{3}\mbox{ and }\frac{1}{2}uG'(u,y)-G(u,y)\leq-cu^{3},\quad\mbox{for all }u\geq0.
    \end{equation}
    The rest of the proof is divided into three steps.

    \textbf{Step 1: Estimate of the $H^{1}$ norm.} We first show that there exists a constant $C_{2}$ such that
    \begin{equation*}
        \mu_{n}:=\|u_{n}\|_{H^{1}(\Omega)}\leq C_{2},\quad\mbox{for all }n\geq1.
    \end{equation*}
    In fact, it follows from \eqref{eq. a PS sequence already given} that
    \begin{equation*}
        \int_{\Omega}u_{n}\cdot \nabla I(u_{n})dxdy\geq-\|\nabla I(u_{n})\|_{H^{-1}(\Omega)}\|u_{n}\|_{H^{1}(\Omega)}\geq-\frac{\mu_{n}}{n}.
    \end{equation*}
    As Lemma~\ref{lem. Sobolev in a channel} implies that $u_{n}\in L^{3}(\Omega)$, recalling \eqref{eq. Legendre transform of G(u,y)} yields
    \begin{equation}\label{eq. PS verification step 1 inequality 1}
        -\frac{\mu_{n}}{2n}\leq\int_{\Omega}\Big\{\frac{|\nabla u_{n}|^{2}}{2}+\frac{1}{2}u_{n}G'(u_{n},y)\Big\}dxdy=I(u_{n})-c\int_{\Omega}u_{n}^{3}dxdy.
    \end{equation}
    
    On the other hand, it follows from \eqref{eq. Legendre transform of G(u,y)} that
    \begin{equation*}
        I(u_{n})\geq\frac{\mu_{n}^{2}}{2}+\int_{\Omega}G(u_{n},y)dxdy\geq\frac{\mu_{n}^{2}}{2}-C\int_{\Omega}u_{n}^{3}dxdy.
    \end{equation*}
    Thus one has \[
\displaystyle\int_{\Omega}u_{n}^{3}dxdy\geq\frac{1}{C}\Big(\frac{\mu_{n}^{2}}{4}-\frac{1}{2}I(u_{n})\Big).
\]
This, together with \eqref{eq. PS verification step 1 inequality 1}, yields
    \begin{equation*}
        -\frac{\mu_{n}}{2n}\leq I(u_{n})-\frac{c}{C}\Big(\frac{h_{n}^{2}}{4}-\frac{1}{2}I(u_{n})\Big).
    \end{equation*}
    Therefore, choosing $C_{2}:=C_{2}(C_{1})$ yields
    \begin{equation*}
        \mu_{n}\leq C\sqrt{I(u_{n})+\frac{1}{n^{2}}}\leq C_{2}.
    \end{equation*}
    
    \textbf{Step 2: Strong convergence in $L^{p}$ for $p>2$.} We next show that there exists a subsequence of $\{u_{n}\}$ (still labeled by $\{u_{n}\}$), which converges strongly in $L^{p}(\Omega)$ for all $p>2$.
        
    Denote $\Omega_{L}=[-L,L]\times[-1,1]$. For $u_{n}\in\overline{\mathcal{C}}$, it follows from Corollary~\ref{lem. L infinity uniform decay} that
    \begin{equation}\label{eq. decay rate for u_n}
        \sup_{\Omega\setminus\Omega_{L}}u_{n}(x,y)\leq\frac{C}{2\sqrt{L}}.
    \end{equation}
    Furthermore, the uniform boundedness of $\mu_{n}=\|u_{n}\|_{H^{1}(\Omega)}$, together with Lemma~\ref{lem. Sobolev in a channel}, implies that $\{u_{n}\}$ is uniformly bounded in $L^{p}(\Omega)$ for each $p\geq2$. Therefore, there exists a subsequence of $\{u_{n}\}$ (still labeled by $\{u_{n}\}$), which strongly converges in $L^{p}(\Omega_{L})$ for each $p\in(1,\infty)$ and any $L>0$.
        
    Now, let us prove the strong convergence in $L^{p}(\Omega)$ for $p>2$. For every $\delta>0$, we choose a sufficiently large $L=L(p,\delta,C_{1})$, then it follow from \eqref{eq. decay rate for u_n} that:
    \begin{equation*}
        \int_{\Omega\setminus\Omega_{L}}|u_{n}|^{p}dxdy\leq\Big(\frac{C}{2\sqrt{L}}\Big)^{p-2}\|u_{n}\|_{L^{2}(\Omega)}^{2}\leq\left(\frac{\delta}{4}\right)^{p},\quad\mbox{for all }n\geq1.
    \end{equation*}
    Notice that $u_{n}\in H^{1}_{0}(\Omega)$ and that $\{u_{n}\}$ is uniformly bounded in $\overline{\Omega_{L}}$. Hence, applying the Rellich compactness theorem to $\{u_{n}\}$ yields that a subsequence of $\{u_{n}\}$ (still labeled by $\{u_{n}\}$) is a Cauchy sequence in $L^{p}(\Omega_{L})$. Therefore, there exists an $N=N(p,\delta,C_{1})>0$, such that
    \begin{equation*}
        \|u_{m}-u_{n}\|_{L^{p}(\Omega_{L})}\leq\delta/2\quad\mbox{for all }m,n>N.
    \end{equation*}
    Thus, for sufficiently large $m,n$, it holds that
    \begin{equation*}
        \|u_{m}-u_{n}\|_{L^{p}(\Omega)}\leq\|u_{m}-u_{n}\|_{L^{p}(\Omega_{L})}+\|u_{m}\|_{L^{p}(\Omega\setminus\Omega_{L})}+\|u_{n}\|_{L^{p}(\Omega\setminus\Omega_{L})}\leq\frac{\delta}{2}+\frac{\delta}{4}+\frac{\delta}{4}=\delta.
    \end{equation*}
    This means that the subsequence $\{u_{n}\}$ converges strongly in $L^{p}(\Omega)$ for all $p>2$.
    
    \textbf{Step 3: Strong $H^{1}$ convergence.} Note that
    \begin{equation*}
        \nabla I(u_{m})-\nabla I(u_{n})=G'(u_{m},y)-G'(u_{n},y)-\Delta(u_{m}-u_{n}).
    \end{equation*}
    It then follows from the triangle inequality that
    \begin{align*}
        &\,\Big|\int_{\Omega}(u_{m}-u_{n})\Big(\nabla I(u_{m})-\nabla I(u_{n})\Big)dxdy\Big|\\
        \leq&\,\|u_{m}-u_{n}\|_{H^{1}(\Omega)}\cdot\|\nabla I(u_{m})-\nabla I(u_{n})\|_{H^{-1}(\Omega)}\\
        \leq&\,\Big(\|u_{m}\|_{H^{1}(\Omega)}+\|u_{n}\|_{H^{1}(\Omega)}\Big)\cdot\frac{m+n}{mn}\leq2C_{3}\frac{m+n}{mn}.
    \end{align*}
    Hence, it holds that
    \begin{equation*}
        \lim_{m,n\to \infty}\int_{\Omega}\Big\{|\nabla(u_{m}-u_{n})|^{2}+\Big(G'(u_{m},y)-G'(u_{n},y)\Big)(u_{m}-u_{n})\Big\}dxdy=0.
    \end{equation*}

    By the assumptions (G1)-(G3), one has $G''(u,y)\geq-Cu$ for $(u,y)\in[0,\infty)\times[-1,1]$, so
    \begin{align*}
        &\int_{\Omega}\Big(G'(u_{m},y)-G'(u_{n},y)\Big)(u_{m}-u_{n})dxdy\\
        =&\int_{\Omega}G''(\theta,y)\cdot(u_{m}-u_{n})^{2}dxdy\geq-C\int_{\Omega}|\theta|\cdot(u_{m}-u_{n})^{2}dxdy,
    \end{align*}
    where $\theta=\theta(x,y)$ satisfies
    \begin{equation*}
        \theta(x,y)\in\Big(\min\{u_{m}, u_n\}(x,y),\max\{u_m,u_{n}\}(x,y)\Big).
    \end{equation*}
    Therefore, for any $\epsilon>0$, there exists an  $N\in\mathbb{N}$, such that for all $m,n\geq N$, one has
    \begin{align*}
        \int_{\Omega}|\nabla(u_{m}-u_{n})|^{2}dxdy\leq&\epsilon+C\int_{\Omega}(1+|\theta|)\cdot(u_{m}-u_{n})^{2}dxdy\\
        \leq&\epsilon+C\|\theta\|_{L^{2}(\Omega)}\|u_{m}-u_{n}\|_{L^{4}(\Omega)}^{2},
    \end{align*}
    where the Cauchy-Schwarz inequality has been used in the last step.
    
    As $\theta$ is bounded between $u_{m}$ and $u_{n}$, one has
    \begin{equation*}
        \|\theta\|_{L^{2}(\Omega)}\leq\|u_{m}\|_{L^{2}(\Omega)}+\|u_{n}\|_{L^{2}(\Omega)}\leq C\Big(\|u_{m}\|_{H^{1}(\Omega)}+\|u_{n}\|_{H^{1}(\Omega)}\Big)\leq C\cdot2C_{3}=:C_{4}.
    \end{equation*}
 As what has already been shown in Step 2,  $\|u_{m}-u_{n}\|_{L^{4}(\Omega)}^{2}$ tends to zero by the strong convergence of $\{u_{n}\}$ in $L^{4}(\Omega)$. Hence one has
    \begin{equation*}
\limsup_{m,n\to\infty}\int_{\Omega}|\nabla(u_{m}-u_{n})|^{2}dxdy\leq\limsup_{m,n\to\infty}6C_{4}\|u_{m}-u_{n}\|_{L^{4}(\Omega)}^{2}=0.
    \end{equation*}
    This implies that $\{u_{n}\}$ strongly converges in $H^{1}(\Omega)$.
\end{proof}

 Part (2) of Proposition~\ref{prop. Existence and convergence of Palais-Smale sequence} is a direct consequence of Lemma~\ref{lem. Palais-Smale}. Therefore, the first part of Theorem~\ref{MainTheorem} is proved. 
 
The second part of Theorem~\ref{MainTheorem} is proved in Section \ref{Section4}.

\section{Rigidity of analytic steady states in the channel}\label{Section4}
 In this section, we show that there is a dense family of uniformly non-stagnant shear flows $\varphi$,  which every analytic steady state asymptotically converging to at the far fields of the channel must be the shear flows themselves. We denote by 
    $\mathcal{A}(\varphi)$ be the class of analytic solution to \eqref{eqn:Steady in channel}.

\begin{proposition}\label{prop:dense-rigidity}
Let $\varphi_0:[-1,1]\to\mathbb{R}$ be a uniformly non-stagnant shear profile satisfying $\varphi_0'(y)>0$. Suppose that $\mathcal{A}(\varphi_0)$ contains a non-shear element. Then there exists a constant $\delta(\varphi_0)>0$ such that for every $\epsilon\in (0, \delta(\varphi_0))$, $\mathcal{A}(\varphi_{0}+\epsilon y)$ has a unique element $\varphi_{0}(y)+\epsilon y$.
\end{proposition}
It is easy to see that the second part of  Theorem~\ref{MainTheorem} follows directly from Proposition~\ref{prop:dense-rigidity}.
One of the key ingredients of Proposition~\ref{prop:dense-rigidity} is to show that there is a semilinear elliptic equation governing analytic solutions of \eqref{eqn:Steady in channel}.
\subsection{Semilinear equation for the stream function of analytic flows}\label{sec3}
The purpose of this section is to show that the stream function of an analytic steady state in the
infinite channel, whose far field is a uniformly non-stagnant shear flow, should satisfy a global semilinear elliptic equation.
The proof can be regarded as an adaptation of that in 
\cite{ElgindiHuangSaidXie_ClassificationSteadyEulerFlows_DMJ}, where the analytic flows in bounded simply connected domains with slip boundary conditions are studied. 

\begin{proposition}[Semilinear equation]\label{prop. semilinear equation}
Given $\varphi \in C^2([-1,1])$ satisfying $\varphi'(y)>0$ for $y\in [-1,1]$, 
let $\psi$ be an analytic steady state in the channel $\Omega =\mathbb{R}\times[-1,1]$, and assume that
\begin{equation}\label{eq. psi converge to varphi in L infty}
\|\psi(x,\cdot)-\varphi(\cdot)\|_{C^{2}([-1,1])}\to0 \text{ as } |x|\to+\infty.
\end{equation}
 Then there exists a continuous function $F$ such that
\[
\Delta\psi=F(\psi).
\]
Furthermore, if the range of $\psi$ is $[\check{A},\hat{B}]$, then $F$ and $\varphi$ satisfy the following properties.
\begin{itemize}
    \item[(1)] $F$ is analytic at $(\check{A},\hat{B})$. 
    \item[(2)] If $F$ is not analytic near $\check{A}$, then every point in $\psi^{-1}(\check{A})$ is a critical point of vanishing degree $2k_0\geq 4$, i.e., for any $x\in \psi^{-1}(\check{A})$, one has $\nabla^k \psi(x)=0$ for any integer $k$ satisfying $1\leq k < 2k_0$; furthermore, in this case there exist a sequence $\{a_k\}_{k=k_0-1}^\infty$ such that
    \begin{equation}\label{eqn:Puiseuxa}
        F(s)=\sum_{k=k_0-1}^\infty a_k (s-\check{A})^{\frac{k}{k_0}}\quad\text{with}\,\, a_{k_{0}-1}\neq0.
    \end{equation}
\item[(3)] If $F$ is not analytic near $\hat{B}$, then every point in $\psi^{-1}(\hat{B})$ is a critical point of vanishing degree $2k_1\geq 4$,  i.e., for any $x\in \psi^{-1}(\hat{B})$, one has $\nabla^k \psi(x)=0$ for any integer $k$ satisfying $1\leq k < 2k_1$; furthermore, in this case there exist a sequence $\{b_k\}_{k=k_1-1}^\infty$ such that
\begin{equation}\label{eqn:Puiseuxb}
    F(s)=\sum_{k=k_1-1}^\infty b_k (\hat{B}-s)^{\frac{k}{k_1}}\quad \text{with}\,\,b_{k_{1}-1}\neq0.
\end{equation}
\item[(4)] $\varphi \in C^{\omega}([-1,1])$.
\end{itemize}
\end{proposition}

We now apply the argument developed in
Sections 2-4 in \cite{ElgindiHuangSaidXie_ClassificationSteadyEulerFlows_DMJ}.
We sketch the proof and explain how the argument can be adapted to the
present unbounded domain.
\begin{proof}
Denote by $C_{\psi}=\{(x,y)\in\Omega:\nabla\psi(x,y)=0\}$ to be the set of critical points of $\psi$. It follows from \eqref{eq. psi converge to varphi in L infty} that all the critical points of $\psi$ are bounded.

By the structure theorem for the zero set of the analytic functions (see \cite{krantz2002primer,sullivan2006combinatorial}), the critical set $C_{\psi}$ is a union of a finite number of closed loops (together with the portion of $\partial \Omega$) and finite isolated points. As the connected component of  $C_{\psi}$ must be a level set of $\psi$,  and $\psi$ takes different values at the upper and lower boundary of $\Omega$, it follows that any connected component of $C_{\psi}$ cannot contain the points both in the upper boundary and lower boundary of $\Omega$.
Moreover, Lemma 3.4 in \cite{ElgindiHuangSaidXie_ClassificationSteadyEulerFlows_DMJ} implies that if the critical set $C_{\psi}$ intersects either the upper or the lower boundary of $\Omega$, the portion of $\partial \Omega$ between the two intersection points has to lie in $C_{\psi}$. Now in particular, the critical set $C_{\psi}$ gives the following decomposition for the channel $\Omega$. 
\begin{figure}[h!]\label{fig:div}
 \caption{An illustration for the division of $\Omega$}
    \centering
\includegraphics[width=0.5\linewidth]{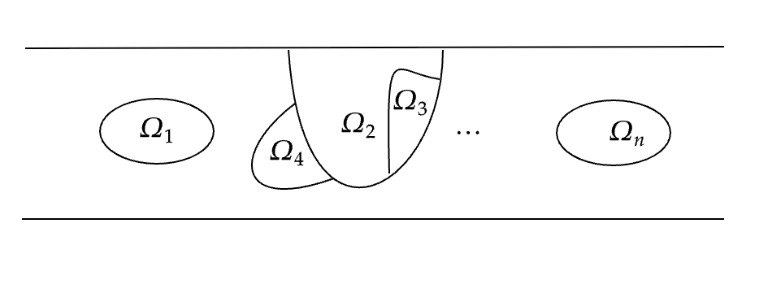}
\end{figure} 
In figure \ref{fig:div}, $\Omega_1,\cdots,\Omega_{n}$ are bounded sets enclosed by the critical points of $\psi$ and there is no closed curve consisting of critical points of $\psi$ in $\Omega_0:=\Omega\setminus \cup_{i=1}^{n} \overline{\Omega_{i}}$, which is path-connected.

First, as $\Omega_0$ is open, connected and contains no closed curve consisting of critical points of $\psi$(in particular, $\Omega_0\setminus C_{\psi}$ is path-connected ), it follows from the proof of  Lemma 1.3 in \cite{ElgindiHuangSaidXie_ClassificationSteadyEulerFlows_DMJ}, there is a continuous function $F_0$ such that $\Delta\psi=F_0(\psi)$ in  $\Omega_0$. Let $[\check{A}_0,\hat{B}_0]$ be the range of $\psi$ in $\overline{\Omega_0}$, by Lemma 2.2 and Lemma 4.4 in \cite{ElgindiHuangSaidXie_ClassificationSteadyEulerFlows_DMJ},  $F_0$ admits Puiseux series expansion in $[\check{A}_0,\hat{B}_0]$

As $\psi$ is analytic, by the unique continuation property of analytic functionss, the condition \eqref{eq. psi converge to varphi in L infty} implies that $\psi$ is nonradial in $\Omega_{i}$  for each $i\in\{1,\cdots,n\}$. By Theorem 1.1 in \cite{ElgindiHuangSaidXie_ClassificationSteadyEulerFlows_DMJ}, there is a function $F_{i}$, such that $\Delta \psi= F_{i}(\psi)$ in each $\Omega_{i}$. 
Let $[\check{A}_i,\hat{B}_i]$ be the range of $\psi$ in $\overline{\Omega_{i}}$, using a similar argument as in $\Omega_{0}$, $F_{i}$ admits Puiseux series expansion in $[\check{A}_i,\hat{B}_i]$. Moreover, as we mentioned that $\psi$ is non-radial in $\Omega_{i}$, by Proposition 4.9 in  \cite{ElgindiHuangSaidXie_ClassificationSteadyEulerFlows_DMJ}, $F_{i}$ is analytic at the trace value of $\psi$ at $\Omega_i$. In particular,  $\Delta \psi -F_{i}(\psi)$ is analytic near $\partial \Omega_i$. Now as $\Delta \psi=F_i(\psi)$ in $\Omega_i$, the unique continuation property of analytic functions implies that $F_0(\psi)=\Delta \psi=F_{i}(\psi)$ near $\partial \Omega_i$ in $\Omega_0$. 

As a result, $F_0$ extends to be a global Puiseux function $F$ and that the value of $F$ restricted to $[\check{A}_i,\hat{B}_i]$ agrees with $F_{i}$. In particular, one has
\begin{equation*}
    \Delta \psi=F(\psi)\quad\mbox{in }\Omega.
\end{equation*}

Denote
\[
\check{A}=\min\{\check{A}_i\}_{i=0}^n\quad \text{and}\quad \hat{B}=\max\{\hat{B}_i\}_{i=0}^n.
\] 
Suppose that $F$ is not analytic at $s\in(\check{A},\hat{B})$, by the continuity of $\psi$ and the structure theorem of the zero set of the analytic functions, the level set $\{\psi=s\}$ contains a closed loop $\Gamma$. Let $\Omega_{\Gamma}^\sharp$ be the simply-connected domain enclosed by $\Gamma$. Notice that $\Delta \psi=F(\psi)$ in $\Omega_{\Gamma}^\sharp$, and $F$ is not analytic at $s$, which is the trace value of $\psi$ in $\Omega_{\Gamma}^\sharp$. This contradicts Proposition 4.9 in \cite{ElgindiHuangSaidXie_ClassificationSteadyEulerFlows_DMJ}. Therefore, $F$ is analytic in $(\check{A},\hat{B})$.
 
    Now $F$ can only lose analyticity at $\check{A}$ or $\check{B}$, which are the global minimum and maximum of $\psi$ in $\Omega$. 
 In the case where $F$ is not analytic at $\check{A}$, similarly as before, the preimage of $\check{A}$ cannot contain a closed loop. By the structure theorem for zero set for analytic function,  the preimage of $\check{A}$ consists of  union of  finitely many isolated critical points. 
 Moreover, Lemma 4.4 in \cite{ElgindiHuangSaidXie_ClassificationSteadyEulerFlows_DMJ} gives the Puiseux expansion of $F$ near $\check{A}$ in the form of equation \eqref{eqn:Puiseuxa}.
 Similarly, if $F$ is not analytic near $\hat{B}$,  the preimage of $\hat{B}$ consists of  union of  finitely many isolated critical points and $F$ admits a Puiseux expansion in the form of equation \eqref{eqn:Puiseuxb} near $\hat{B}$. 
 In particular, we have shown $\Delta \psi=F(\psi)$, with $F$ analytic in $[\varphi(-1),\varphi(1)]$.
 It follows from\eqref{eq. psi converge to varphi in L infty} that one has $\varphi^{''}(y)=F(\varphi)$. Furthermore, the regularity theory of elliptic equations  (see, for example, \cite{Morrey1958Analyticity}) implies that $\varphi \in C^{\omega}([-1,1])$.
\end{proof}
  We now perform a refined analysis of analytic solutions to the semilinear elliptic equation to finish the proof of Proposition \ref{prop:dense-rigidity}.
\subsection{Analysis of the Semilinear Equations}
Based on the maximum principle, we have the following lemma.

\begin{lemma}\label{lem: proof via strong maximum principle}
Given $a\in (-\infty, -1]$ and $b\in [1,\infty)$,  let $\widetilde{\varphi}:[a,b]\to\mathbb{R}$ be a $C^2$ function such that
\begin{equation*}
    \widetilde{\varphi}'(t) > 0, \qquad 
    \widetilde{\varphi}''(t) = F(\widetilde{\varphi}(t)), 
    \quad \text{for all } t\in[a,b],
\end{equation*}
where $F$ is a continuous function on $[\tilde{\varphi}(a), \tilde{\varphi}(b)]$. 
Let $\psi$ be a $C^2$ solution to the following problem:
\begin{equation}\label{eq:semilinear-f}
    \left\{\begin{aligned}
        &\Delta\psi = F(\psi)\quad\mbox{in }\Omega = \mathbb{R}\times(-1,1),\\
        &\psi(x,\pm1)\equiv\widetilde{\varphi}(\pm1),\\
        &\|\psi(x,\cdot)-\widetilde{\varphi}(\cdot)\|_{C^{2}([-1,1])}\to0,\quad\mbox{as } |x|\to+\infty.
    \end{aligned}\right.
\end{equation}
If the range of $\psi$ is contained in $[\widetilde{\varphi}(a),\widetilde{\varphi}(b)]$, then $\psi(x,y)\equiv \widetilde{\varphi}(y)$.
\end{lemma}
\begin{proof}
Define a scalar function $\tilde{h}(x,y)\in [a,b]$ implicitly by
\begin{equation*}
    \widetilde{\varphi}\big(\tilde{h}(x,y)\big) = \psi(x,y).
\end{equation*}
Since $\widetilde{\varphi}'(t)>0$ and the range of $\psi$ lies in $(\widetilde{\varphi}(a),\widetilde{\varphi}(b))$, the function $\tilde{h}$ is well-defined. Denote $h(x,y):=\tilde{h}(x,y)-y$.  By the implicit function theorem, one has $h\in C^2(\Omega)$. Furthermore, the straightforward computations yield 
\begin{align*}
    &\psi_x = \widetilde{\varphi}'(\widetilde{h}) h_x, 
    \qquad \psi_y = \widetilde{\varphi}'(\widetilde{h}) (1 + h_y), \\
    &\psi_{xx} = \widetilde{\varphi}''(\widetilde{h}) h_x^2 + \widetilde{\varphi}'(\widetilde{h}) h_{xx}, \quad \psi_{yy} = \widetilde{\varphi}''(\widetilde{h}) (1 + h_y)^2 + \widetilde{\varphi}'(\widetilde{h}) h_{yy},
\end{align*}
and 
\begin{align*}
    \Delta\psi 
    &= \psi_{xx} + \psi_{yy} = \widetilde{\varphi}'(\widetilde{h})\Delta h 
       + \widetilde{\varphi}''(\widetilde{h})\bigl[h_x^2 + (1+h_y)^2\bigr].
\end{align*}
Using the equation \eqref{eq:semilinear-f} and the identity $\widetilde{\varphi}''(\tilde{h})=F(\widetilde{\varphi}(\tilde{h}))=F(\psi)$ gives $\Delta\psi = \widetilde{\varphi}''(\widetilde{h})$. Therefore, one has the following identity
\[
\widetilde{\varphi}''(\widetilde{h})
    = \widetilde{\varphi}'(\widetilde{h})\Delta h 
      + \widetilde{\varphi}''(\widetilde{h})\bigl[h_x^2 + (1+h_y)^2\bigr].
\]
Rearranging the above identity gives the following second order elliptic equation  
\[
\Delta h + \mathbf{V}\cdot\nabla h = 0,\quad\mbox{where }\mathbf{V} := \frac{\widetilde{\varphi}''(\widetilde{h})}{\widetilde{\varphi}'(\widetilde{h})} 
\bigl(h_x,\, 2+h_y\bigr)\in C^{1}(\Omega).
\]

By construction, $h(x,\pm1)\equiv0$ and $h(x,y)\to 0$ as $|x|\to+\infty$. Applying the strong maximum principle yields $h\equiv 0$, and hence $\psi(x,y)\equiv \widetilde{\varphi}(y)$.
\end{proof}
An immediate consequence of Lemma \ref{lem: proof via strong maximum principle} is the following comparison principle for analytic solutions of \eqref{eqn:Steady in channel}.
\begin{proposition}\label{prop:rigidity-from-range}
In the setting of Proposition~\ref{prop:dense-rigidity}, let $\widetilde{\varphi}$ denote the maximal analytic extension of $\varphi$ that remains monotone increasing. If the range of every element in $\mathcal{A}(\varphi)$ is contained in the range of $\widetilde{\varphi}$, then $\mathcal{A}(\varphi)=\{\varphi\}$.
\end{proposition}

Section~2 shows that the range of elements in $\mathcal{A}(\varphi)$ is not necessarily contained in the range of $\widetilde{\varphi}$. In the remainder of this section, we identify sufficient conditions on $\varphi$ under which the range of an element in $\mathcal{A}(\varphi)$ is contained in the range of  $\widetilde{\varphi}$. 

We perform a refined analysis of the asymptotic data $\varphi$ and its analytic continuation. We begin by introducing the following notions.

\begin{definition}[Maximal interval]\label{def. maximal interval}
Let $\varphi\colon [-1,1]\to\mathbb{R}$ be an analytic function satisfying $\varphi'(y)>0$ for $y\in[-1,1]$. An open interval $(a,b)$ is called the \emph{maximal interval} of $\varphi$ if the following statements hold.
\begin{itemize}
    \item[(1)] $\varphi$ admits a unique analytic monotone extension $\widetilde \varphi$ in $(a,b)$, such that $\widetilde{\varphi}\equiv \varphi$ in $[-1,1]$ and $\widetilde \varphi'(y)>0$ for all $y\in(a,b)$;
    \item[(2)] If there exists an interval $(a',b')$ such that $\varphi$ satisfies the properties in \textup{(1)} in $(a',b')$, then $(a',b')\subseteq (a,b)$.
\end{itemize}
\end{definition}

We now define the notion of symmetry-breaking condition at the endpoints.

\begin{definition}[Admissible asymptotic shear flow]\label{def. symmetry breaking}
Let $\varphi$ be analytic in $[-1,1]$ and satisfy $\varphi'(y)>0$ for $y\in [-1,1]$, let $(a,b)$ be the maximal interval of $\varphi$, and let $\widetilde \varphi$ be its analytic extension. We say that $\varphi$  \emph{is an admissible asymptotic shear flow} if, at each endpoint, one of the following corresponding conditions holds. 
\begin{itemize}
 \item[(1)]
 $b<+\infty$ ($a>-\infty$, respectively) and $\widetilde{\varphi}$ is not $C^2$ near $b$ ($a$, respectively).
    \item[(2)] $b<+\infty$ ($a>-\infty$, respectively) is a critical endpoint of $\widetilde \varphi$, and $\widetilde \varphi$ admits an analytic extension beyond $b$ ($a$, respectively) that is not even with respect to $b$ ($a$, respectively).
    \item[(3)] $b<+\infty$ ($a>-\infty$, respectively) is a critical endpoint of $\widetilde \varphi$, and $\widetilde \varphi$ admits no analytic extension beyond $b$ ($a$, respectively).
    \item[(4)] $b<+\infty$ ($a>-\infty$, respectively) is not a critical endpoint of $\widetilde \varphi$, and $\widetilde \varphi$ admits no analytic extension beyond $b$ ($a$, respectively).
    \item[(5)] $b=+\infty$ ($a=-\infty$, respectively) and
    \[
    \lim_{y\to+\infty}\widetilde \varphi(y)=+\infty
    \qquad
    (\lim_{y\to-\infty}\widetilde \varphi(y)=-\infty, \text{ respectively}).
    \]
\end{itemize}
\end{definition}
\begin{remark}
    The behavior of $\widetilde{\varphi}$ at $a$ and at $b$ is allowed to be different. For example, if $\widetilde{\varphi}$ satisfies case (1) at $a$ and satisfies case (2) at $b$ in Definition~\ref{def. symmetry breaking}, then $\varphi$ is still admissible.
\end{remark}
From Definition~\ref{def. symmetry breaking}, one has the following ``openness" of the admissible condition. The proof is left to the reader.
\begin{lemma}\label{rmk. admissible is open}
    Assume that $(a,b)$ is the maximal interval of $\varphi$, and that $(a_{\epsilon},b_{\epsilon})$ is the maximal interval of $\varphi_{\epsilon}=\varphi+\epsilon y$ with $0<\epsilon\ll1$. If $\varphi$ is admissible at $a$ (or $b$), then $\varphi_{\epsilon}$ is admissible at $a_{\epsilon}$ (or $b_{\epsilon}$).
\end{lemma}
\begin{lemma}\label{thm. the rigidity theorem with symmetry breaking condition}
    Assume that $\varphi(y)$ is analytic in $[-1,1]$ with $\varphi'(y)>0$ for $y\in [-1,1]$, and let $(a,b)$ be its maximal interval, $\widetilde{\varphi}$ be the maximal monotone extension. If $\varphi$ is an admissible asymptotic shear flow and $\psi$ is an analytic solution of \eqref{eqn:Steady in channel}, then the range of $\psi$ (say, $[\check{A},\hat{B}]$) is contained in the range of $\widetilde{\varphi}$, and  $\widetilde{\varphi}^{'}(y)>0$ for all $y$ satisfying $\widetilde{\varphi}(y)\in [\check{A},\hat{B}]$.
\end{lemma}
\begin{proof}
Let $[\check{A},\hat{B}]$ be the range of $\psi$ in $\Omega$. We first show that $\hat{B}\leq \widetilde \varphi(b)$. Assume, for contradiction, that $\hat{B}>\widetilde \varphi(b)$, then $\widetilde{\varphi}(b)\in(\check{A},\hat{B})$.

Let $F_{\psi}$ denote the nonlinearity in the semilinear elliptic equation satisfied by $\psi$ in Proposition~\ref{prop. semilinear equation}, as $F_{\psi}$ and  $\widetilde \varphi''\circ \varphi^{-1}$ are Puiseux functions that agree on $[\varphi(-1),\varphi(1)]$. It follows from the unique-continuation property for Puiseux functions that $\widetilde \varphi$ satisfies
\begin{equation}\label{eq. widetilde f ODE}
    \widetilde \varphi''=F_{\psi}(\widetilde \varphi),\quad\mbox{whenever }\widetilde \varphi(t)\in[\check{A},\hat{B}].
\end{equation}
Let $G_{\psi}$ be an anti-derivative of $F_{\psi}$. It follows from Proposition~\ref{prop. semilinear equation} that $F_{\psi}$ is analytic near $\widetilde{\varphi}(b)$. Consequently, the quantity
\begin{equation}\label{eq. constant quantity from 1st integral}
    (\widetilde \varphi')^2(t)-2G_{\psi}(\widetilde \varphi(t))
\end{equation}
is constant for all $t\in [a,b]$, and $\widetilde \varphi$ is $C^1$ near $b$. Therefore, in view of \eqref{eq. widetilde f ODE}, we may analytically continue $\widetilde \varphi$ beyond $b$ and the  admissible condition implies that $b$ is a critical point for $\widetilde{\varphi}$. By the local well-posedness theory for ODEs, this continuation is necessarily even with respect to $b$, contradicting the admissible asymptotic shear condition.

Note that \eqref{eq. widetilde f ODE} implies that $\widetilde \varphi$ is $C^2$ whenever $\check{A}\leq \widetilde \varphi \leq\hat{B}$. In particular, if $\widetilde \varphi$ is not $C^2$, then necessarily
\[
b>\widetilde \varphi^{-1}(\hat{B}),\quad\mbox{and hence }\widetilde \varphi'\bigl(\widetilde \varphi^{-1}(\hat{B})\bigr)>0.
\]

It remains to treat the case where $\hat{B}=\widetilde \varphi(b)$ and $\widetilde \varphi$ is $C^2$ near $b$. We claim that in this case $\widetilde \varphi'(b)>0$. Assume, on the contrary, $\widetilde \varphi'(b)=0$. As above, the quantity \eqref{eq. constant quantity from 1st integral} is constant for all $t$ near $b$. By Proposition~\ref{prop. semilinear equation}, there exists a $k_0\geq 2$ such that
\begin{equation*}
(\widetilde \varphi')^2=\sum_{k=2k_0-1}^{\infty} a_k \bigl(\hat{B}-\widetilde \varphi\bigr)^{\frac{k}{k_0}}\quad\mbox{near }\hat{B},\quad\mbox{where }a_{2k_{0}-1}\neq0.
\end{equation*}
Define $\displaystyle\eta := (\hat{B}-\widetilde \varphi)^{\frac{1}{2k_0}}$. Then, for $|y-b|\ll 1$, one has
\begin{equation}\label{eqn:ODE}
\left\{\begin{aligned}
    &\eta'
    =-\frac{1}{2k_0}\sqrt{\sum_{k=2k_0-1}^{\infty} a_k \eta^{2(k-2k_0+1)}},\\
    &\eta(b)=0,
\end{aligned}\right.
\end{equation}
where the right-hand side is analytic and symmetric with respect to $0$. Hence $\eta$ admits an analytic extension as a solution to \eqref{eqn:ODE}. By the local well-posedness theory for ODEs, $\eta$ is odd with respect to $b$. Consequently, $\hat{B}-\eta^{2k_0}$ gives an analytic extension of $\widetilde \varphi$ beyond $b$, and this extension is even with respect to $b$. This again contradicts the assumption that $\varphi$ is an admissible asymptotic shear flow.
The case at the left endpoint $\check{A}$ can be treated similarly. This completes the proof of Lemma \ref{thm. the rigidity theorem with symmetry breaking condition}.
\end{proof}
An immediate corollary of Lemmas \ref{lem: proof via strong maximum principle} and  \ref{thm. the rigidity theorem with symmetry breaking condition} is the following corollary. 
\begin{corollary}\label{cor:symmetry breaking implies rigidity}
    If the asymptotic shear flow $\varphi$ satisfies the admissible condition in Definition~\ref{def. symmetry breaking}, then every analytic steady state is necessarily a shear flow, i.e., $\mathcal{A}(\varphi)=\{\varphi\}$. 
\end{corollary}
We now complete the proof of Proposition~\ref{prop:dense-rigidity}.
\begin{proof}[Proof of Proposition~\ref{prop:dense-rigidity}]
If $\varphi$ is of the form $\varphi(t)=a_0+a_1t+a_2t^2$, then for every $\epsilon$, the function $\varphi_\epsilon(t):=\varphi(t)+\epsilon t$ satisfies $\varphi_\epsilon''(t)=2a_2$. Hence the corresponding nonlinearity is simply $F\equiv 2a_2$. It is straightforward to check that the only solution to the problem
\begin{equation*}
    \left\{\begin{aligned}
        &\Delta\psi=2a_{2}&\mbox{in }&\mathbb{R}\times(-1,1),\\
        &\psi(x,\cdot)\to \varphi_{\epsilon}(\cdot)&\mbox{as }&|x|\to+\infty,\\
        &\psi(x,\pm1)\equiv \varphi_{\epsilon}(\pm1)
    \end{aligned}\right.
\end{equation*}
is $\psi(x,y)\equiv \varphi_\epsilon(y)$.

Thus, in order to prove the proposition, it suffices to show that whenever $\varphi$ is non-admissible and is not of the form $a_0+a_1t+a_2t^2$, then for every sufficiently small $\epsilon>0$, $\varphi_\epsilon(t):=\varphi(t)+\epsilon t$ is admissible.

Let $(a,b)$ be the maximal interval of $\varphi$, and let $(a_\epsilon,b_\epsilon)$ be the maximal interval of $\varphi_\epsilon$. We divide the argument into the following cases.

    \textbf{Case 1: $a=-\infty$ and $b=\infty$.} Since $\varphi$ is non-admissible, its analytic extension $\widetilde{\varphi}$ must satisfy at least one of the following properties:
    \[
    \lim_{t\to-\infty}\widetilde \varphi(t)>-\infty
    \quad\text{or }\lim_{t\to\infty}\widetilde \varphi(t)<\infty.
    \]
    On the other hand, by Lemma \ref{rmk. admissible is open}, for $\widetilde \varphi_\epsilon(t)=\widetilde \varphi(t)+\epsilon t$ with $\epsilon>0$, we still have
    \[
    a_\epsilon=-\infty,
    \quad
    b_\epsilon=\infty,\quad\mbox{and }\lim_{t\to\pm\infty}\widetilde \varphi_\epsilon(t)=\pm\infty.
    \]
    Hence $\varphi_\epsilon$ is admissible.

    \textbf{Case 2: $a>-\infty$ and $b=\infty$.} Since $\widetilde \varphi$ is monotonically increasing, Lemma \ref{rmk. admissible is open} implies that, at the right endpoint, it holds that
    \[
    b_\epsilon=\infty\mbox{ and }\lim_{t\to\infty}\widetilde \varphi_\epsilon(t)=\infty,\quad\mbox{for }\epsilon>0,
    \]
    and that the admissible condition is always satisfied on the right endpoint for $\widetilde\varphi_{\ep}$. At the left endpoint $a$, as $\widetilde{\varphi}$ is not admissible, $\varphi$ is $C^2$ near $a$.  Note that $a$ has to be a critical point of $\varphi$. Otherwise, $\varphi$ would admit a monotone analytic extension beyond $a$, which contradicts $(a,b)$ being maximal interval. As $\widetilde{\varphi}$ is not admissible,  $\widetilde \varphi$ admits an analytic extension  $\widetilde \varphi^*$ that is even with respect to $a$. As $\varphi$ is not quadratic, there exists an even integer $k\geq 4$ such that
\begin{equation}\label{eqn:vanishing degree even}
    (\widetilde{\varphi^*})^{(1)}(a)=(\widetilde{\varphi^*})^{(j)}(a)=  0\mbox{ for } 3\leq j\leq k-1,\quad\mbox{and }(\widetilde{\varphi^a})^{(k)}(a)\neq 0.
\end{equation}

        Therefore, for every sufficiently small $\epsilon>0$, the function $\widetilde \varphi_\epsilon(t)=\widetilde \varphi(t)+\epsilon t$ admits an analytic extension beyond the left endpoint $a_\epsilon$, which is a critical point of $\widetilde \varphi_\epsilon$. Moreover, by \eqref{eqn:vanishing degree even}, one has $\displaystyle\lim_{\epsilon \rightarrow 0}|a_{\epsilon}-a|=0$, $a_{\epsilon}\neq a$, and thus
        \begin{equation*}
            (\widetilde{\varphi^a})^{(k-1)}(a_{\epsilon})\neq 0.
        \end{equation*}
        Therefore, this extension is not even with respect to $a_\epsilon$. In particular, one has $a_\epsilon>-\infty$.
    Hence $\varphi_\epsilon$ is admissible.
    
    \textbf{Case 3: $a=-\infty$ and $b<\infty$.} This is analogous to Case 2 and we omit the details.

    \textbf{Case 4: $a>-\infty$ and $b<\infty$.} Based on Lemma \ref{rmk. admissible is open}, as the admissible condition of the far field is stable under the perturbation $\epsilon y$, if the maximal interval of $\varphi_{\epsilon}$ becomes $(a_{\epsilon},b_{\epsilon})$, then one can analyze the endpoint $a_{\epsilon}$ as in Case 2, and $b_{\epsilon}$ as in Case 3. This shows that $\varphi_{\epsilon}$ is admissible for $\epsilon>0$ sufficiently small.
    
In all cases, we conclude that for every non-admissible $\varphi$ which is not quadratic, the perturbed profile $\varphi_\epsilon(y)=\varphi(y)+\epsilon y$ is admissible for all sufficiently small $\epsilon>0$. Hence by Corollary~\ref{cor:symmetry breaking implies rigidity}, the proof of Proposition~\ref{prop:dense-rigidity} is completed.
\end{proof}
\appendix
\section{Sobolev inequalities}\label{AppendixA}
In this appendix, we prove several Sobolev inequalities for two-dimensional functions, which play an important role in the proof of main results and may be useful for other problems.

\begin{lemma}[Sobolev inequality with measure]\label{lem. Sobolev with measure}
    If $u\in H^{1}(\mathbb{R}^{2})$ is compactly supported, then there exists a constant $C$ such that
    \begin{equation*}
        C\int_{\mathbb{R}^{2}}|\nabla u|^{2}dxdy\geq\frac{\int_{\mathbb{R}^{2}}|u|^{2}dxdy}{\mathcal{H}^{2}(\{(x,y)\in\mathbb{R}^{2}:u(x,y)\neq0\})}.
    \end{equation*}
    Here, $\mathcal{H}^{2}$ is the two-dimensional Hausdorff measure.
\end{lemma}
\begin{proof}
    Without loss of generality, we assume that $u$ is smooth with compact support. We start with the standard $W^{1,1}$ Sobolev inequality (with $n=2$ and $p=1$)
    \begin{equation*}
        C\int_{\mathbb{R}^{2}}|\nabla v|dxdy\geq\Big(\int_{\mathbb{R}^{2}}|v|^{2}dxdy\Big)^{1/2}.
    \end{equation*}
    By setting $v=u^{2}$ and applying the Cauchy-Schwarz inequality, one has
    \begin{align*}
        \Big(\int_{\mathbb{R}^{2}}|u|^{4}dxdy\Big)^{1/2}\leq&C\int_{\mathbb{R}^{2}}|\nabla v|dxdy\leq2C\int_{\mathbb{R}^{2}}|u\nabla u|dxdy\\
        \leq&2C\Big(\int_{\mathbb{R}^{2}}|u|^{2}dxdy\Big)^{1/2}\Big(\int_{\mathbb{R}^{2}}|\nabla u|^{2}dxdy\Big)^{1/2}.
    \end{align*}
    This means that
    \begin{equation*}
        C\int_{\mathbb{R}^{2}}|\nabla u|^{2}dxdy\geq\frac{\int_{\mathbb{R}^{2}}|u|^{4}dxdy}{\int_{\mathbb{R}^{2}}|u|^{2}dxdy}\geq\frac{\int_{\mathbb{R}^{2}}|u|^{2}dxdy}{\int_{\mathbb{R}^{2}}\chi_{\{u\neq0\}}dxdy},
    \end{equation*}
    where we have used the Cauchy-Schwarz inequality again in the last step.
\end{proof}

Next, for functions in $H^{1}_{0}(\Omega)=H^{1}_{0}\Big(\mathbb{R}\times(-1,1)\Big)$, we have the following $L^{p}$ estimate.
\begin{lemma}[$L^{p}$-Sobolev inequalities]\label{lem. Sobolev in a channel}
    Let $u\in H^{1}_{0}(\Omega)$ and let $p\geq2$. Then there exists a uniform constant $\mathfrak{C}(p)$, such that
    \begin{equation}\label{eq. L p sobolev desired inequality}
        \|u\|_{L^{p}(\Omega)}\leq\mathfrak{C}(p)\Big(\int_{\Omega}|\nabla u|^{2}dxdy\Big)^{1/2}.
    \end{equation}
\end{lemma}
\begin{proof}
    The case $p=2$ is trivial, as we have
    \begin{equation*}
        \int_{-1}^{1}|u(x^{*},\eta)|^{2}d\eta\leq C_{2}\int_{-1}^{1}|\partial_{y}u(x^{*},\eta)|^{2}d\eta
    \end{equation*}
    for each $x^{*}\in\mathbb{R}$. The rest of the proof is devoted to the case $p>2$.
    
    Define $V(h):=\mathcal{H}^{2}(\Omega\cap\{|u|>h\})$. Then one has
    \begin{equation*}
        \int_{0}^{\infty}ph^{p-1}V_{h}dh=\int_{\Omega}|u|^{p}dxdy.
    \end{equation*}
    Denote $\Pi_{h}$ by
    \begin{equation*}
        \Pi_{h}:=\Big\{x\in\mathbb{R}:(\{x\}\times[-1,1])\cap\{|u|>h\}\neq\emptyset\Big\}.
    \end{equation*}
    It is easy to see that $\mathcal{H}^{1}(\Pi_{h})\geq\frac{1}{2}V(h)$. For each $x^{*}\in\Pi_{h}$, suppose that $|u(x^{*},y^{*})|>h$ for some $y^{*}\in(-1,1)$. Hence, on the vertical line $\{x^{*}\}\times[-1,1]$,
    \begin{align*}
        \int_{-1}^{1}|\nabla u(x^{*},\eta)|^{2}d\eta\geq&\int_{-1}^{y^{*}}|\partial_{y}u(x^{*},\eta)|^{2}d\eta+\int_{y^{*}}^{1}|\partial_{y}u(x^{*},\eta)|^{2}d\eta\\
        \geq&\frac{1}{y^{*}+1}\Big|\int_{-1}^{y^{*}}\partial_{y}u(x^{*},\eta)d\eta\Big|^{2}+\frac{1}{1-y^{*}}\Big|\int_{y^{*}}^{1}\partial_{y}u(x^{*},\eta)d\eta\Big|^{2}\\
        \geq&\frac{h^{2}}{y^{*}+1}+\frac{h^{2}}{1-y^{*}}\geq2h^{2}.
    \end{align*}
    Therefore, we have
    \begin{equation}\label{eq:estgradient_A}
        \int_{\Omega}|\nabla u|^{2}dxdy\geq\int_{\Pi_{h}\times[-1,1]}|\nabla u(x,y)|^{2}dxdy\geq2h^{2}\mathcal{H}^{1}(\Pi_{h})\geq h^{2}V(h).
    \end{equation}
    Denote $\displaystyle E=\int_{\Omega}|\nabla u|^{2}dxdy$. It follows from \eqref{eq:estgradient_A} that one has
    \begin{equation*}
        \int_{0}^{2\sqrt{E}}ph^{p-1}V(h)dh\leq\int_{0}^{2\sqrt{E}}ph^{p-3}Edh=\frac{2^{p-2}}{p-2}E^{p/2}.
    \end{equation*}

    On the other hand, by the coarea formula and the isoperimetric inequality, it holds that
    \begin{align*}
        E\geq&\int_{\{|u|>\sqrt{E}\}}|\nabla u|^{2}dxdy=\int_{\sqrt{E}}^{\infty}\int_{\partial\{|u|=h\}}|\nabla u|d\mathcal{H}^{1}(\partial\{|u|=h\})dh\\
        \geq&\int_{\sqrt{E}}^{\infty}\frac{\Big[\int_{\partial\{|u|=h\}}d\mathcal{H}^{1}(\partial\{|u|=h\})\Big]^{2}}{\int_{\partial\{|u|=h\}}|\nabla u|^{-1}d\mathcal{H}^{1}(\partial\{|u|=h\})}dh\\
        =&\int_{\sqrt{E}}^{\infty}\frac{\Big[\mathcal{H}^{1}(\partial\{|u|=h\})\Big]^{2}}{-V'(h)\Big|_{t=h}}dh\geq4\pi\int_{\sqrt{E}}^{\infty}\frac{V(h)}{-V'(h)}dh.
    \end{align*}
   As $V(h)\geq 0$ and $-V'(h)\geq 0$,
    \begin{equation}\label{eqn:entropy}
    \begin{aligned}
         \frac{E}{4\pi}\ln{\left(\frac{1}{V(H)}\right)}\geq &\, \frac{E}{4\pi}\int_{\sqrt{E}}^{H}\frac{d}{dh}\ln{\left(\frac{1}{V(h)}\right)}dh\\
          \geq &\,\Big\{\int_{\sqrt{E}}^{H}\frac{V(h)}{-V'(h)}dh\Big\}\Big\{\int_{\sqrt{E}}^{H}\frac{-V'(h)}{V(h)}dh\Big\}.
    \end{aligned}
    \end{equation}
    Hence for $H\geq2\sqrt{E}$, one has
    \begin{equation*}
        \frac{E}{4\pi}\ln{\Big(\frac{1}{V(H)}\Big)}\geq\Big(\int_{\sqrt{E}}^{H} 1 dh\Big)^2\geq \frac{H^2}{4}.
    \end{equation*}
    In other words, one has $V(H)\leq e^{-\frac{\pi}{E}H^{2}}$. Thus it holds that
    \begin{equation*}
        \int_{2\sqrt{E}}^{\infty}ph^{p-1}V(h)dh\leq\int_{2\sqrt{E}}^{\infty}ph^{p-1}e^{-\frac{\pi}{E}h^{2}}dh=E^{p/2}\int_{2}^{\infty}p t^{p-1}e^{-\pi t^{2}}dt.
    \end{equation*}
    Therefore, choosing
    \[
    \mathfrak{C}(p)=\frac{2^{p-2}}{p-2}+\int_{2}^{\infty}p t^{p-1}e^{-\pi t^{2}}dt
    \]  yields the desired inequality \eqref{eq. L p sobolev desired inequality}.
\end{proof}

Applying a similar method gives the following $L^{\infty}$ decay estimate for functions in $\overline{\mathcal{C}}$.
\begin{corollary}
[Uniform $L^{\infty}$ decay]\label{lem. L infinity uniform decay}
    For $u(x,y)\in\overline{\mathcal{C}}$, which is defined right after \eqref{defmathcalC}, one has
    \begin{equation*}
        \sup_{|x|\geq L,\ y\in[-1,1]}u(x,y)\leq\frac{\|u\|_{H^{1}(\Omega)}}{\sqrt{4L}}.
    \end{equation*}
\end{corollary}
\begin{proof}
    It suffices to notice that each $\Pi_{h}$ defined in the proof of Lemma~\ref{lem. Sobolev in a channel} is connected and is symmetric about $\{x=0\}$ for $u\in\overline{\mathcal{C}}$.
\end{proof}

\section{Analysis of the heat equation}\label{AppendixB}
In this appendix, we provide a detailed analysis of the heat equation. The analysis plays a crucial role in the construction of the Palais-Smale sequence.
\begin{lemma}\label{lem. appendix well-posedness} 
  Assume that $g(v, y)\in C^{\infty}(\mathbb{R}\times [-1,1])$, $g(0,y)\equiv0$, and the partial derivative with respect to $v$ satisfies
    \begin{equation*}
        |g'(v,y)|\leq C(1+|v|)\quad\mbox{for all }(v,y)\in\mathbb{R}\times[-1,1].
    \end{equation*}   Consider the initial-boundary value problem for the following parabolic equation
    \begin{equation}\label{eq. Heat Cauchy problem}
        \left\{\begin{aligned}
            &\partial_{t}v=\Delta v+g(v,y),\mbox{ when }(t,x,y)\in(0,T)\times\Omega,\\
            &v(0,x,y)= u(x,y),\\
            &v(t,x,\pm1)=0\mbox{ and }\lim_{x\to\pm\infty}v(t,x,y)=0.
        \end{aligned}\right..
    \end{equation}
For $\Omega=\mathbb{R}\times (-1,1)$,  let $u\in C^{2,\alpha}_{loc}(\Omega)$ be compactly supported and satisfy
    \begin{equation}\label{eq. compatible assump}
        u=\Delta u=0\quad\mbox{on }\partial\Omega=\mathbb{R}\times\{\pm1\}.
    \end{equation}
    Then the following statements hold.
    \begin{itemize}
        \item[(1)] \textbf{(Existence):} There exists a time $T>0$ depending on the initial data $u(x,y)$, such that the problem \eqref{eq. Heat Cauchy problem} has a classical solution $v$ with $v(t,\cdot) \in H^{2}(\Omega)$ and $\partial_{t}v(t,\cdot) \in L^{2}(\Omega)$ for all $t\in[0,T]$.
        \item[(2)] \textbf{(Uniqueness and stability):} Let $\widetilde{u}\in C^{2,\alpha}_{loc}(\Omega)$ be another compactly supported function satisfying \eqref{eq. compatible assump}, and let $\widetilde{v}$ be a classical solution of \eqref{eq. Heat Cauchy problem} with the initial data $\widetilde{u}$. Assume that $\|v\|_{L^{\infty}([0,T]\times\Omega)}\leq M$ for some $M>0$. For any $\epsilon>0$, there exists a $\delta=\delta(\epsilon,M,T)>0$ such that if
        \begin{equation*}
            \|u-\widetilde{u}\|_{L^{\infty}(\Omega)}+\|u-\widetilde{u}\|_{L^{2}(\Omega)}+\|u-\widetilde{u}\|_{H^{1}(\Omega)}\leq\delta,
        \end{equation*}
        then $\|v-\widetilde{v}\|_{H^{1}(\{t\}\times\Omega)}+\|v-\widetilde{v}\|_{L^{\infty}(\{t\}\times\Omega)}\leq\epsilon$ for all $t\in[0,T]$.
        In particular, the classical solution to \eqref{eq. Heat Cauchy problem} with the same initial data $u$ is unique.
        \item[(3)] \textbf{(Symmetry):} If the initial data $u(x,y)\in\mathcal{C}$, which is defined in \eqref{defmathcalC}, then $v(t,\cdot)\in\overline{\mathcal{C}}$ for all $t>0$.
    \end{itemize}
\end{lemma}
\begin{proof}
    The proof is divided into seven steps.

    \textbf{Step 1: Domain exhaustion and the approximation problem.} Assume for simplicity that $u(x,y)$ is supported in $[-L,L]\times[-1,1]$. For every $R>L$, we consider the following family of exhaustion domains $\mathcal{S}_{R}\subseteq\Omega$, such that
    \begin{equation*}
        [-R,R]\times[-1,1]\subseteq\mathcal{S}_{R}\subseteq\Omega.
    \end{equation*}
 Furthermore, we require that $\partial\mathcal{S}_{R}\in C^{2,\alpha}$ and that $\mathcal{S}_{R}$ is convex. Denote by ${\bnu}$ the outward unit normal of $\partial\mathcal{S}_{R}$. Then the curvature vector of $\partial\mathcal{S}_{R}$ takes the form
    \begin{equation*}
        \bH(x,y)=-h(x,y){\bnu}(x,y),\quad\mbox{where }h(x,y)\geq0\mbox{ for }(x,y)\in\partial\mathcal{S}_{R}.
    \end{equation*}
    
    We consider the following approximating problem of \eqref{eq. Heat Cauchy problem}:
    \begin{equation}\label{eq. Heat approximation problem}
        \left\{\begin{aligned}
            &\partial_{t}v_{R}=\Delta v_{R}+g(v_{R},y), &(x,y,t)\in \mathring{\mathcal{S}}_{R}\times(0,+\infty),\\
            &v_{R}(0,x,y)=u(x,y),  &(x,y)\in {\mathcal{S}}_{R},\\
            &v_{R}(t,x,y)=0,&(x,y)\in\partial\mathcal{S}_{R},
        \end{aligned}\right.
    \end{equation}
    where $\mathring{\mathcal{S}}_{R}$ is the interior of ${\mathcal{S}}_{R}$.
    Since $u$ satisfies \eqref{eq. compatible assump} and is supported in $\mathcal{S}_{R}$, the compatibility condition holds at $\partial\mathcal{S}_{R}\cap\{t=0\}$. By the standard well-posedness theory, there exists a $\overline{t}>0$, such that \eqref{eq. Heat approximation problem} has a classical solution $v_{R}(t,x,y)$ for $t\in(0,\overline{t})$. Next, we need to provide uniform estimates independent of the radius $R$.

    \textbf{Step 2: $L^{\infty}$ control.} By the assumptions on $g(u,y)$, one has
    \begin{equation}\label{eq:LinftyvR_A}
        \big|(\partial_{t}-\Delta)v_{R}\big|\leq C(|v_{R}|+|v_{R}|^{2}).
    \end{equation}
    Let $\varrho(t)$ solve  \begin{equation*}
      \left\{
      \begin{aligned}
          &\frac{d\varrho}{d t}=C(\varrho+\varrho^2),\\
          & \varrho(0)=\|u\|_{L^{\infty}(\Omega)},
      \end{aligned}
      \right.
    \end{equation*}
    where $C$ is the same as the one in \eqref{eq:LinftyvR_A}.
    Then there exists a $T>0$ depending on $\|u\|_{L^{\infty}(\Omega)}$, such that $|\varrho(t)|\leq2\|u\|_{L^{\infty}(\Omega)}$ for all  $t\in[0,T]$. Now by the maximum principle, it holds that 
    \begin{equation*}
        \|v_{R}(t,\cdot)\|_{L^{\infty}(\mathcal{S}_{R})}\leq\varrho(t)\leq2\|u\|_{L^{\infty}(\Omega)}.
    \end{equation*}
    Therefore, there exists a uniform constant $K=K(u)>0$, such that
    \begin{equation*}
        |g'(w,y)|\leq K\quad\mbox{for all }|w|\leq2\|u\|_{L^{\infty}(\Omega)}.
    \end{equation*}

    \textbf{Step 3: $L^{\infty}(0,T;L^{2}(\mathcal{S}_{R}))$ control of $v_{R}$ and $\partial_{t}v_{R}$.} Using the condition that $v_{R}\equiv0$ on $\partial\mathcal{S}_{R}$ and integrating by parts yields that for all $t\in[0,T]$, one has
    \begin{align*}
        \frac{d}{dt}\int_{\{t\}\times\mathcal{S}_{R}}\frac{v_{R}^{2}}{2}dxdy=&\int_{\{t\}\times\mathcal{S}_{R}}v_{R}\cdot\partial_{t}v_{R}dxdy=\int_{\{t\}\times\mathcal{S}_{R}}v_{R}\cdot(\Delta v_{R}+g)dxdy\\
        =&\int_{\{t\}\times\mathcal{S}_{R}}(v_{R}\cdot g-|\nabla v_{R}|^{2})dxdy\leq\int_{\{t\}\times\mathcal{S}_{R}}K\cdot v_{R}^{2}dxdy.
    \end{align*}
    In other words, for all $t\in[0,T]$,
    \begin{equation*}
        \int_{\{t\}\times\mathcal{S}_{R}}|v_{R}|^{2}dxdy\leq e^{2Kt}\int_{\mathcal{S}_{R}}u^{2}dxdy\leq e^{2KT}\int_{\mathbb{R}^{2}\times[-1,1]}u^{2}dxdy=C(K,T,u)<\infty.
    \end{equation*}
    
    Differentiating \eqref{eq. Heat approximation problem} in $t$ and multiplying by $\partial_{t}v_{R}$ gives
    \begin{align*}
        &\frac{d}{dt}\int_{\{t\}\times\mathcal{S}_{R}}\frac{|\partial_{t}v_{R}|^{2}}{2}dxdy=\int_{\{t\}\times\mathcal{S}_{R}}\partial_{t}v_{R}\Big(\Delta\partial_{t}v_{R}+\frac{\partial g}{\partial v}\cdot\partial_{t}v_{R}\Big)dxdy\\
        \leq&\int_{\{t\}\times\mathcal{S}_{R}}\Big(-|\nabla\partial_{t}v_{R}|^{2}+K\cdot|\partial_{t}v_{R}|^{2}\Big)dxdy\leq K\int_{\{t\}\times\mathcal{S}_{R}}|\partial_{t}v_{R}|^{2}dxdy,\quad\mbox{for all }t\in[0,T],
    \end{align*}
    where we used the boundary condition $\partial_{t}v_{R}\equiv0$ on $\partial\mathcal{S}_{R}$. Similarly, one has
    \begin{align*}
        &\int_{\{t\}\times\mathcal{S}_{R}}|\partial_{t}v_{R}|^{2}dxdy\leq e^{2Kt}\int_{\{0\}\times\mathcal{S}_{R}}|\partial_{t}v_{R}|^{2}dxdy\\
        \leq&e^{2KT}\int_{\mathbb{R}^{2}\times[-1,1]}|\Delta u+g(u,y)|^{2}dxdy=C(K,T,u,g)<\infty,\quad\mbox{for all }t\in[0,T].
    \end{align*}

    \textbf{Step 4: $L^{\infty}(0,T;H^{2}(\mathcal{S}_{R}))$ control of $v_{R}$.} Multiplying both sides of \eqref{eq. Heat approximation problem} by $\Delta v_{R}$ and using $|g(v_{R},y)|\leq K|v_{R}|$ yields that for all $t\in[0,T]$, one has
    \begin{align*}
        \int_{\{t\}\times\mathcal{S}_{R}}(\Delta v_{R})^{2}dxdy=&\,\int_{\{t\}\times\mathcal{S}_{R}}\Delta v_{R}\cdot(\partial_{t}v_{R}-g)dxdy\\
        \leq&\,\int_{\{t\}\times\mathcal{S}_{R}}\frac{(\Delta v_{R})^{2}}{2}dxdy+\int_{\{t\}\times\mathcal{S}_{R}}\Big(|\partial_{t}v_{R}|^{2}+K^{2}|v_{R}|^{2}\Big)dxdy.
    \end{align*}
    This, together with the estimates obtained in Step 3, implies that
    \begin{equation*}
        \int_{\{t\}\times\mathcal{S}_{R}}(\Delta v_{R})^{2}dxdy\leq C(K,T,u,g)<\infty\quad\mbox{for }t\in[0,T].
    \end{equation*}
    Observe that
    \begin{equation}\label{eq. H2 Step 4 boundary 1}
        v_{R}\equiv0\mbox{ and }\Delta v_{R}=\partial_{t}v_{R}+g(v_{R},y)=g(0,y)=0\mbox{ on }\partial\mathcal{S}_{R},
    \end{equation}
    and that the tangential derivatives of $v_{R}$ on $\partial\mathcal{S}_{R}$ are zero. Hence it holds that
    \begin{equation*}
        \partial_{\vec{\bnu}}\frac{|\nabla v_{R}|^{2}}{2}=\langle\nabla_{\bnu}\nabla v_{R},\nabla v_{R}\rangle=\langle\nabla_{\bnu}\nabla v_{R},(\partial_{\bnu}v_{R})\bnu\rangle=\partial_{\bnu}v_{R}\cdot D^{2}v_{R}(\bnu,\bnu)\quad\mbox{on }\partial\mathcal{S}_{R}.
    \end{equation*}
    It follows from \eqref{eq. H2 Step 4 boundary 1} that
    \begin{equation*}
        D^{2}v_{R}(\bnu,\bnu)=\Delta v_{R}+\langle\bH,\nabla v_{R}\rangle-\Delta_{\partial\mathcal{S}_{R}}v=\langle\bH,(\partial_{\bnu}v_{R})\bnu\rangle=-h(x,y)\cdot\partial_{\bnu}v_{R}.
    \end{equation*}
    As $h(x,y)\geq0$ (recall that $\mathcal{S}_{R}$ is convex), we have
    \begin{equation}\label{eq. H2 Step 4 boundary 2}
        \partial_{\bnu}\frac{|\nabla v_{R}|^{2}}{2}=\partial_{\bnu}v_{R}\cdot D^{2}v_{R}(\bnu,\bnu)=-h(x,y)\cdot|\partial_{\bnu}v_{R}|^{2}\leq0.
    \end{equation}
    Using \eqref{eq. H2 Step 4 boundary 1} and \eqref{eq. H2 Step 4 boundary 2} and integrating by parts gives
    \begin{align*}
        \int_{\{t\}\times\mathcal{S}_{R}}|D^{2}v_{R}|^{2}dxdy=&\int_{\{t\}\times\mathcal{S}_{R}}\Big\{\Delta\frac{|\nabla v_{R}|^{2}}{2}-\langle\nabla v_{R},\nabla\Delta v_{R}\rangle\Big\}dxdy\\
        =&\int_{\{t\}\times\mathcal{S}_{R}}(\Delta v_{R})^{2}dxdy+\int_{\{t\}\times\partial\mathcal{S}_{R}}\Big\{\partial_{\vec{\nu}}\frac{|\nabla v_{R}|^{2}}{2}-\Delta v_{R}\cdot\partial_{\vec{\nu}}v_{R}\Big\}d\sigma\\
        \leq&\int_{\{t\}\times\mathcal{S}_{R}}(\Delta v_{R})^{2}dxdy\leq C(K,T,u,g)<\infty.
    \end{align*}

    \textbf{Step 5: Construction of $v$.} Using the uniform bounds obtained above, we can extract a subsequence of $\{v_{R}\}$ and obtain its weak limit $v(t,x,y)$, such that for every $\rho>L$ (recall that $u$ is supported in $[-L,L]\times[-1,1]$) and $\tau\in[0,T]$, one has
    \begin{align*}
        &v_{R_{k}}(\tau,x,y)\to v(\tau,x,y)&\mbox{weakly in }H^{2}(\{\tau\}\times\mathcal{S}_{\rho}),\\
        &\partial_{t}v_{R_{k}}(\tau,x,y)\to\partial_{t}v(\tau,x,y)&\mbox{weakly in }L^{2}(\{\tau\}\times\mathcal{S}_{\rho}).
    \end{align*}
    Clearly, $v(t,x,y)$ is a weak solution to \eqref{eq. Heat Cauchy problem} for $t\in[0,T]$. As the estimates in Steps 3-4 are independent of $\mathcal{S}_{R}$, the weak limit $v(t,x,y)$ must satisfy the same estimates in $\Omega$. Using the Schauder estimate yields that $v(t,x,y)$ is a classical solution to \eqref{eq. Heat Cauchy problem} in $[0,T]\times\Omega$.

    \textbf{Step 6: Uniqueness and stability.} Under the assumption that $\|v\|_{L^{\infty}([0,T]\times\Omega)}\leq M$, it follows from the assumptions on $|g'(v,y)|$ that
    \begin{equation*}
        \big|(\partial_{t}-\Delta)(v-\widetilde{v})\big|\leq C\cdot(M+1)\cdot|v-\widetilde{v}|\quad\mbox{in }[0,T]\times\Omega.
    \end{equation*}
    By the maximum principle, one has
    \begin{equation*}
        \frac{d}{dt}\|v-\widetilde{v}\|_{L^{\infty}(\{t\}\times\Omega)}\leq C(M+1)\|v-\widetilde{v}\|_{L^{\infty}(\{t\}\times\Omega)}.
    \end{equation*}
    Given that $\|u-\widetilde{u}\|_{L^{\infty}(\Omega)}\leq\delta$ for some small $\delta$, it holds that
    \begin{equation}\label{eq. L infty stable}
        \|v-\widetilde{v}\|_{L^{\infty}([0,T]\times\Omega)}\leq\frac{\epsilon}{2}.
    \end{equation}
    Therefore, $|\widetilde{v}|\leq2M$ in $[0,T]\times\Omega$. Since
    \begin{equation*}
        \big|g'(v,y)\big|\leq K':=C\cdot(1+2M)\quad\mbox{for }|v|\leq2M,
    \end{equation*}
    it follows that $w:=v-\widetilde{v}$ satisfies
    \begin{equation}\label{eq. w=error equation}
        |(\partial_{t}-\Delta)w|\leq K'|w|\quad\mbox{in }[0,T]\times\Omega.
    \end{equation}
   Similar to Step 3, multiplying by $w$ on both sides of \eqref{eq. w=error equation} gives
    \begin{equation*}
        \|w\|_{L^{2}(\{t\}\times\Omega)}\leq e^{K'T}\delta,\quad\mbox{for all }t\in[0,T].
    \end{equation*}
    Multiplying by $\partial_{t}w$ on both sides of \eqref{eq. w=error equation} and using the fact $\partial_{t}w=0$ on $\mathbb{R}\times\{\pm1\}$ yields
    \begin{align*}
        &\int_{\{t\}\times\Omega}\Big\{|\partial_{t}w|^{2}+\partial_{t}(\frac{|\nabla w|^{2}}{2})\Big\}dxdy=\int_{\{t\}\times\Omega}\Big\{|\partial_{t}w|^{2}-\partial_{t}w\cdot\Delta w\Big\}dxdy\\
        \leq&\int_{\{t\}\times\Omega}K'|\partial_{t}w|\cdot|w|dxdy\leq\int_{\{t\}\times\Omega}\Big\{|\partial_{t}w|^{2}+(K')^{2}|w|^{2}\Big\}dxdy.
    \end{align*}
    Therefore, for $t\in[0,T]$, one has
    \begin{equation}\label{eq. H1 stable}
    \begin{aligned}
        \int_{\{t\}\times\Omega}|\nabla w|^{2}dxdy\leq&\,\|u-\widetilde{u}\|_{H^{1}(\Omega)}^{2}+2(K')^{2}T\sup_{t\in[0,T]}\|w\|_{L^{2}(\{t\}\times\Omega)}^{2}\\
\leq&\,\Big(1+2(K')^{2}Te^{2KT}\Big)\delta^{2}\leq\frac{\epsilon^{2}}{4}.
\end{aligned}
    \end{equation}
    Therefore, from \eqref{eq. L infty stable} and \eqref{eq. H1 stable}, the $H^{1}(\Omega)\cap L^{\infty}(\Omega)$-stability and uniqueness of \eqref{eq. Heat Cauchy problem} are established.

    \textbf{Step 7: Symmetry of the solution.} If the initial data $u(x,y)\in\mathcal{C}$, it remains to prove that $v_{R}(t,\cdot,\cdot)\in\overline{\mathcal{C}}$ for all $R$. To see this, we first realize that $v_{R}\equiv0$ is a solution of \eqref{eq. Heat approximation problem} with zero initial data. As $u(x,y)\geq0$, we apply the maximum principle to get that $v_{R}\geq0$ everywhere. Next, for each $\lambda\leq0$, define
    \begin{align*}
        T_{\lambda}=&\{(t,x,y):\ t\in[0,T],\ x=\lambda,\ y\in[-1,1]\},\\
        \Sigma_{\lambda}=&\{(t,x,y):\ t\in[0,T],\ x<\lambda,\ y\in[-1,1]\}.
    \end{align*}
    Let $x_{\lambda}:=2\lambda-x$ be the reflection of $x$ about the plane $T_{\lambda}$. Define $v_{R,\lambda}(t,x,y)=v_{R}(t,x_{\lambda},y)$ for each $(t,x,y)\in\Sigma_{\lambda}$ and let
    \begin{equation*}
        w_{R,\lambda}(t,x,y)=v_{R,\lambda}(t,x,y)-v_{R}(t,x,y)=v_{R}(t,2\lambda-x,y)-v_{R}(t,x,y).
    \end{equation*}
    Straightforward computations give
    \begin{equation*}
        (\partial_{t}-\Delta)w_{R,\lambda}(t,x,y)=g(v_{R,\lambda}(t,x,y),y)-g(v_{R}(t,x,y),y)=\frac{\partial g}{\partial u}(\Theta(t,x,y),y)w_{R,\lambda}(t,x,y),
    \end{equation*}
    where $\Theta(t,x,y)$ is a function bounded between $v_{R}(t,x,y)$ and $v_{R,\lambda}(t,x,y)$. As
    \begin{equation*}
        \|v_{R}\|_{L^{\infty}([0,T]\times\mathcal{S}_{R})}\leq C(T,R,u,g)\quad\mbox{for every fixed }T,
    \end{equation*}
    one has
    \begin{equation*}
        (\partial_{t}-\Delta)w_{R,\lambda}(t,x,y)=c_{R,\lambda}(t,x,y)w_{R,\lambda}(t,x,y),
    \end{equation*}
    where $c_{R,\lambda}$ is bounded in every finite time interval. Clearly, $w_{R,\lambda}$ satisfies
    \begin{equation*}
        w_{R,\lambda}(0,x,y)=u(2\lambda-x,y)-u(x,y)\geq0\mbox{ in }\{t=0\}\cap\Sigma_{\lambda},
    \end{equation*}
    and
    \begin{equation*}
        w_{R,\lambda}\equiv0\mbox{ on }T_{\lambda}\cup\{y=\pm1\},\quad w_{R,\lambda}\geq0\mbox{ on }([0,T]\times\partial\mathcal{S}_{R})\cap\Sigma_{\lambda}.
    \end{equation*}
    Applying the maximum principle yields that $w_{R,\lambda}\geq0$ in $\Sigma_{\lambda}$. Passing to the limit gives
    \begin{equation*}
        v(2\lambda-x,y,t)\geq v(x,y,t),\quad\mbox{for every }\lambda\leq0\mbox{ and every }(x,y,t)\in\Sigma_{\lambda}.
    \end{equation*}
    Therefore, we have $v(t,\cdot)\in\overline{\mathcal{C}}$.

    This completes the proof of Lemma~\ref{lem. appendix well-posedness}.
\end{proof}

Next, we consider the so called $\epsilon$-regularity of the heat flow.
\begin{lemma}\label{lem. epsilon regularity}
    Let $0<T\leq1$ and $M\geq1$ be arbitrary. Let $w=w(t,x,y)$ be a nonnegative function, such that $w(0,\cdot)\equiv0$, $w(t,\cdot)\in H^{1}_{0}(\Omega)$ for each $t\in[0,T]$, and
    \begin{equation*}
        \partial_{t}w-\Delta w\leq C_{0}w^{2}+M w+M\quad\mbox{in }[0,T]\times\Omega.
    \end{equation*}
    Then there exists an $\epsilon=\epsilon(C_{0})>0$ independent of $T$ and $M$, such that if
    \begin{equation*}
        \|w\|_{L^{\infty}(0,T,L^{2}(\Omega))}:=\sup_{t\in[0,T]}\|w(t,\cdot)\|_{L^{2}(\Omega)}\leq\epsilon,
    \end{equation*}
    then one has $\displaystyle\sup_{t\in[0,T]}\|w(t,\cdot)\|_{L^{\infty}(\Omega)}\leq M$.
\end{lemma}
\begin{proof}
    Let us denote
    \begin{equation*}
        w_{k}=(w-M_{k})_{+},\quad\mbox{where }M_{k}=(1-2^{-k})M.
    \end{equation*}
    For such a sequence $\{w_{k}\}$, denote
    \begin{equation*}
        A_{k}=\int_{0}^{T}\int_{\Omega}|w_{k}|^{2}dxdydt,\quad\mathcal{D}_{k}:=supp(w_{k})\subseteq[0,T]\times\Omega,
    \end{equation*}
    and
    \begin{equation*}
        \chi_{k}(t,x,y)=\chi_{\mathcal{D}_{k}}=\left\{\begin{aligned}
            &1,&\mbox{if }&(t,x,y)\in\mathcal{D}_{k},\\
            &0,&\mbox{if }&(t,x,y)\notin\mathcal{D}_{k}.
        \end{aligned}\right.
    \end{equation*}
    Obviously, $A_{k}\leq A_{k-1}$. Moreover, it is easy to verify that $w_{k}$ satisfies
    \begin{equation}\label{eq. truncated heat equation}
        \partial_{t}w_{k}-\Delta w_{k}\leq C_{0}w_{k}^{2}+\alpha_{k}w_{k}+\beta_{k},\quad\mbox{where }\alpha_{k}=7M\chi_{k}\mbox{ and }\beta_{k}=5M^{2}\chi_{k}.
    \end{equation}
    Multiplying both sides of \eqref{eq. truncated heat equation} by $w_{k}$ and integrating by parts yields
    \begin{align*}
        &\int_{\Omega}|w_{k}(t,\cdot)|^{2}dxdy+\int_{0}^{t}\int_{\Omega}|\nabla w_{k}|^{2}(s,\cdot)dxdyds\\
        \leq&\int_{0}^{t}\int_{\Omega}\Big(C_{0}w_{k}^{3}(s,\cdot)+\alpha_{k}w_{k}^{2}(s,\cdot)+\beta_{k}w_{k}(s,\cdot)\Big)dxdyds,\quad\mbox{for all }t\in[0,T],
    \end{align*}
    where the boundary terms vanish as $w(t,\cdot)\in H^{1}_{0}(\Omega)$ and $w(0,\cdot)\equiv0$.
    
    In each time slice, we apply Lemma~\ref{lem. Sobolev in a channel} with $p=4$ to get
    \begin{align*}
        \int_{\Omega}C_{0}w_{k}^{3}(t,\cdot)dxdy\leq&C_{0}\Big(\int_{\Omega}w_{k}^{2}(t,\cdot)dxdy\Big)^{1/2}\Big(\int_{\Omega}w_{k}^{4}(t,\cdot)dxdy\Big)^{1/2}\\
        \leq&C_{0}\epsilon\cdot\mathfrak{C}(4)^{2}\int_{\Omega}|\nabla w_{k}(t,\cdot)|^{2}dxdy\leq\frac{1}{2}\int_{\Omega}|\nabla w_{k}(t,\cdot)|^{2}dxdy,
    \end{align*}
    where the smallness of $\epsilon$ has been used in the last step. Then, for all $t\in [0,T]$, one has
    \begin{equation}\label{eqn: De Giohi bound L2}
        \int_{\Omega}|w_{k}|^{2}(t,\cdot)dxdy\leq\int_{0}^{t}\int_{\Omega}\Big(\alpha_{k}w_{k}^{2}(s,\cdot)+\beta_{k}w_{k}(s,\cdot)\Big)dxdyds,
    \end{equation}
    and
    \begin{equation}\label{eq. De Giorgi bound of H1}
        \int_{0}^{T}\int_{\Omega}|\nabla w_{k}|^{2}(t,\cdot)dxdydt\leq2\int_{0}^{T}\int_{\Omega}\Big(\alpha_{k}w_{k}^{2}(t,\cdot)+\beta_{k}w_{k}(t,\cdot)\Big)dxdydt.
    \end{equation}
    
    With the help of Lemma~\ref{lem. Sobolev with measure}, it follows from \eqref{eq. De Giorgi bound of H1} that
    \begin{equation*}
        C\cdot\mathcal{R}\geq\int_{0}^{T}\frac{\int_{\Omega}|w_{k}|^{2}dxdy}{\int_{\Omega}\chi_{k}dxdy}dt,\quad\mbox{where }\mathcal{R}:=\int_{0}^{T}\int_{\Omega}\Big(\alpha_{k}w_{k}^{2}(t,\cdot)+\beta_{k}w_{k}(t,\cdot)\Big)dxdydt.
    \end{equation*}
    Hence, applying the Cauchy-Schwarz inequality to the definition of $A_k$ yields
    \begin{equation}\label{eq:Ak1_B}
    \begin{aligned}
        A_{k}\leq&\Big\{\Big(\sup_{t\in[0,T]}\int_{\Omega}|w_{k}|^{2}dxdy\Big)\cdot\Big(\int_{0}^{T}\frac{\int_{\Omega}|w_{k}|^{2}dxdy}{\int_{\Omega}\chi_{k}dxdy}dt\Big)\Big\}^{1/2}\cdot|\mathcal{D}_{k}|^{1/2}\\
        \leq&\Big\{\mathcal{R}\cdot C\mathcal{R}\Big\}^{1/2}\cdot|\mathcal{D}_{k}|^{1/2}\leq C\mathcal{R}|\mathcal{D}_{k}|^{1/2}.
        \end{aligned}
    \end{equation}
    Moreover, by the Cauchy-Schwarz inequality, one has
    \begin{equation}\label{est_R_B}
        \mathcal{R}\leq\sup|\alpha_{k}|\cdot A_{k}+\sqrt{A_{k}}\cdot\Big(\int_{0}^{T}\int_{\Omega}\beta_{k}^{2}dxdydt\Big)^{1/2}\leq7M\cdot A_{k}+5M^{2}\sqrt{A_{k}}|\mathcal{D}_{k}|^{1/2}.
    \end{equation}
    
    Since $w_{k-1}\geq2^{-k}M$ in $\mathcal{D}_{k}$, it follows that
    \begin{equation}\label{est:Dk_B}
    \displaystyle|\mathcal{D}_{k}|\leq\frac{4^{k}}{M^{2}}A_{k-1}.\end{equation}
    Then combining \eqref{eq:Ak1_B}-\eqref{est:Dk_B} yields
    \begin{align*}
    A_{k}\leq&C\mathcal{R}|\mathcal{D}_{k}|^{1/2}\leq C(7M\cdot A_{k}+5M^{2}\sqrt{A_{k}}|\mathcal{D}_{k}|^{1/2})\cdot|\mathcal{D}_{k}|^{1/2}\\
        \leq&C\sqrt{A_{k}}(2^{k}\sqrt{A_{k}A_{k-1}}+4^{k}A_{k-1}).
    \end{align*}
    Using the fact that $A_{k}\leq A_{k-1}$, we have for $k\geq2$:
    \begin{equation*}
        \sqrt{A_{k}}\leq C(2^{k}\sqrt{A_{k}A_{k-1}}+4^{k}A_{k-1})\leq C\cdot4^{k}A_{k-1}.
    \end{equation*}
    Choose $\epsilon\leq \frac{1}{256C^2}$ Since $T\leq1$, the inequality $\sqrt{A_{k}}\leq C\cdot4^{k}A_{k-1}$  implies $A_{k}\leq \epsilon 4^{-2k}$. In particular,  we have $\displaystyle\lim_{k\rightarrow \infty} A_{k}=0$. As a result, $w\leq M$ almost everywhere in $[0,T]\times\Omega$.

    This completes the proof of Lemma~\ref{lem. epsilon regularity}.
\end{proof}

\medskip 

 {\bf Acknowledgements.}
The research of Huang is partially supported by EPSRC Horizon Europe Guarantee EP/X020886/1. The research of Xie is partially supported by National Key R\&D Program of China 2024YFA1013302, NSFC grants  12426203 and 12571238.

\end{document}